\documentclass[11pt, twoside, reqno]{article}

\usepackage{smart-eqn}
\newcommand{\bu}{\boldsymbol{u}}
\newcommand{\bv}{\boldsymbol{v}}
\newcommand{\bq}{\boldsymbol{q}}

\usepackage{enumerate}
\renewcommand{\div}{\operatorname{div}}
\newcommand{\curl}{\operatorname{curl}}
\newcommand{\bB}{\boldsymbol{\operatorname{ B}}}

\newcommand{\R}{\mathbb{R}}
\renewcommand{\SS}{\mathbb{S}}
\newcommand{\dx}{\,\textup{d}x}
\newcommand{\dt}{\,\textup{d}t}
\newcommand{\dsig}{\,\textup{d}\sigma}
\newcommand{\pdt}{\partial_t}
\newcommand{\ddt}{\frac{\textup{d}}{\textup{d}t}}

\usepackage{autonum}
\usepackage{xcolor}
\usepackage{amsthm}
\usepackage{amsmath}
\usepackage{amssymb}
\usepackage{mathtools}
\usepackage{comment}
\usepackage[margin=.99in]{geometry}

\newtheorem{theorem}{Theorem}
\newtheorem{lemma}{Lemma}
\newtheorem{proposition}{Proposition}

\newtheorem*{assumption*}{Assumptions}
\newtheorem{assumption}{Assumption}

\newtheorem{remark}{Remark}
\numberwithin{lemma}{section}
\numberwithin{proposition}{section}
\numberwithin{theorem}{section}
\numberwithin{equation}{section}
\usepackage{xcolor}
\usepackage{framed}
\usepackage{tocloft}

\begin{document}

\title{{$L^{p}$--$L^{q}$} existence for the open compressible MHD system. }
\date{}
\author{Mostafa Meliani}

\maketitle
\medskip

\centerline{Institute of Mathematics of the Academy of Sciences of the Czech Republic}

\centerline{\v Zitn\' a 25, CZ-115 67 Praha 1, Czech Republic}

	\begin{abstract}
		\noindent We study the local existence of solutions to the Navier--Stokes--Fourier-magnetohydrodynamics (NSF-MHD) system describing the motion of a compressible, viscous, electrically and heat conducting fluid in the $L^p$--$L^q$ class with inhomogeneous boundary conditions. The open system is allowed to receive incoming matter from the outside through (part of) the boundary which we refer to as an inflow boundary. This setup brings about a difficulty in estimating the regularity of the density $\varrho$ which we remedy by assuming appropriate hypotheses on the velocity field, domain boundary and on the boundary and initial data of $\varrho$. The main result ensures the local well-posedness of the full NSF-MHD system which is shown through a linearization combined with a Banach fixed-point theorem.
	\end{abstract}


\tableofcontents

\section{Introduction}
\subsection{Motivation}
There exist by now a wide literature on the existence of strong solutions to equations of compressible isothermal and non-isothermal fluids, including in the case where they are coupled with magnetohydrodynamics; see, e.g.,~\cite{cho2006existence,kagei2006,tang2016strong,valli1982existence} and the references contained therein. There exist, on the other hand, little literature on strong solutions to Navier--Stokes system in the case where there is matter incoming from the exterior; i.e., when the velocity is allowed to be pointing inwards on the boundary ($\bu\cdot n<0$). 
The main difficulty in this case resides in controlling the gradient of the density. The only available results we are aware of are due to Zajaczkowski and Fiszdon~\cite{fiszdon1983initial} in the isothermal case and to Valli and Zajaczkowski~\cite{valli1986navier} in the non-isothermal one. In both aforementioned references, the authors dealt with the existence of strong solutions in a Hilbertian (classical) framework where an inflow boundary is allowed.

We note that in the case where no matter is incoming from the outside into the domain ($\bu \cdot n \geq 0$ on the boundary), the literature is richer and we refer the reader to, e.g., \cite{abbatiello2024local, danchin2001local,danchin2010solvability,kotschote2012strong,solonnikov1980solvability,tang2016local} for local $L^p$--$L^p$ and $L^p$--$L^q$ existence results. We mention here also the following works establishing the existence of $\mathcal{R}$-bounded solution operators and proving the local well-posedness in the maximal $L^p$--$L^q$ regularity class of the two-phase incompressible MHD problem~\cite{frolova2022local,frolova2022maximal}.

We propose in this work to fill this gap in the literature by building an $L^p$--$L^q$ framework to study the full non-isothermal Navier--Stokes system coupled to a magnetohydrodynamics model. The novelty is then twofold: First, establishing the existence and uniqueness of strong solution to an open Navier-Stokes system in the $L^p$--$L^q$ setting. Second, the coupling with equations of evolution of temperature and magnetic field, showing the well-posedness of the complete system.

Our motivation behind this work lies in the fact that the $L^p$--$L^q$ paradigm offers the advantage of allowing more flexibility vis-a-vis, e.g., Sobolev embeddings. Indeed, compared to the Hilbertian (classical) setting, we do not necessarily need to increase the differentiabilty of solutions but rather it may suffice to ask for better $L^q$-integrability. This, in turn, offers the advantage of being able to treat less regular domains (boundary of class $C^2$ in this work compared to $C^3$ in \cite{fiszdon1983initial,valli1986navier} and $C^4$ in \cite{valli1982existence}) and reducing the number of compatibility conditions to be imposed.

Another motivation behind our work lies in our interest in the study of blow-up criteria for the open NSF-MHD system. Indeed, there exist a number of results exploring conditional regularity
in the no-inflow case; see, e.g., \cite{abbatiello2024local,basaric2023conditional,fang2012blow, feireisl2025conditional} and the references contained therein. However, blow-up criteria are not known for the case $\bu\cdot n \not \geq 0$ on the boundary. We believe that the flexibility offered by the $L^p$--$L^q$ setting can help in establishing blow-up criteria in the case of inflow boundary by removing the need for some of the higher-differentiability arguments in, e.g.,~\cite{basaric2023conditional,fang2012blow}. This question will be the subject of study in an upcoming work.

Besides establishing local existence and uniqueness for the NSF-MHD system, this work also deals with minimum and maximum principles for the mass density and temperature in the inflow case which is of physical relevance. I.e., we show that the temperature and density stay positive (physical) as long as the boundary and initial data of the aforementioned quantities are positive and the velocity field is regular enough.

To prove the local solvability of the full system, we first establish $L^p$--$L^q$ estimates of the solutions for a set of linearized equations. We then apply a Banach fixed-point argument. The linearized system is divided into four independent initial-boundary value problems: a linearized continuity equation of hyperbolic nature, and three linearized problems of parabolic nature: one for each of the velocity, temperature and magnetic fields. The Banach fixed-point argument ensures the unique solvability of the coupled nonlinear system. We mention here that since the system of equation is of hyperbolic-parabolic nature, contractivity is only shown in a lower topology. This is due to the well-identified problem of loss of regularity for such systems which prevents the establishing of contractivity estimates in the full space of solutions; see, e.g., \cite{kotschote2012strong_allen-cahn,kotschote2012strong}.

The estimates for the mass density $\varrho$ are obtained by studying the continuity equation for fixed velocity $\bv$. In particular, obtaining estimates for $\nabla_x\varrho$ is made difficult because of the inflow boundary. Indeed, multiplying the differentiated continuity equation by $q|\nabla_x \varrho|^{q-2}\nabla_x \varrho$, gives rise to the term $-\int_{\partial \Omega} \bv \cdot n|\nabla_x \varrho|^q \dsig$ as a result of integration by parts of the term $\bv\cdot \nabla_x\varrho$ in the continuity equation. This boundary term can be omitted for non-slip or outflow boundaries; on the other hand, it should be controlled by the boundary data $\varrho_B$ in the case of the inflow boundary, see the proof of Lemma~\ref{lemma::est_der}. The estimates for the velocity, temperature, and magnetic fields follow from the  $L^p$--$L^q$ theory for the parabolic equations established by Denk, Hieber, and Pr\"uss~\cite{denk2007optimal}.

\subsection{Organization of the paper} The current work is organized as follows: After introducing some necessary notation, we present the system of equations to be studied in Section~\ref{Sec:sys_eq}. In Section \ref{Sec:Main_Result}, we give the assumptions and main result of this work which is contained in Theorem~\ref{theorem:main_thm}. Section~\ref{Sec:MinMax} deals with minimum and maximum principles for the temperature and density. Section~\ref{Sec:Linearized} is devoted to studying a linearized system of equations. The proof of the main Theorem~\ref{theorem:main_thm} is given in Section~\ref{Sec:Main_proof} through a Banach fixed-point argument. Appendix~\ref{App:Appendix} contains technical lemmas used in the proof of Theorem~\ref{theorem:main_thm}. 

\subsection{Notation and functional setting}\label{sec:notation}
Of interest to this work will be the Besov and Triebel--Lizorkin spaces, $B^s_{pq}$ and $F^s_{qp}$, respectively, with $s \in \R$ and $p,q \in (1,\infty)$ . Their relevance resides in the fact that they generalize the classical Sobolev and Sobolev--Slobodeckii spaces $W^{s,p}$. Indeed, we recall that
\begin{equation}
	\begin{aligned}
		W^{s,p} = B^s_{pp}, && \qquad &\textrm{if $s$ noninteger}, \\ 
		W^{m,p} = F^m_{p2}, && \qquad &\textrm{if $m$ integer}. \\ 
	\end{aligned}
\end{equation}
Besov spaces also play an important role as interpolation spaces between Sobolev spaces; see, e.g., \cite[Chapter 6]{bergh2012interpolation}. For convenience of the reader, we recall the definition of the spaces $B^s_{pq}$ and $F^s_{pq}$~\cite[Section 2.3]{triebel1983theory}:
\begin{equation}
	\begin{aligned}
		&B^s_{pq}(\R^n) = \{f \in S'(\R^n) \,| \, \|2^{sj} F^{-1} \phi_j F f\|_{l_q(L^p(\R^n))} < \infty\},\\
		&F^s_{pq}(\R^n) = \{f \in S'(\R^n) \,| \, \|2^{sj} F^{-1} \phi_j F f\|_{L_p(\R^n,\, l_q)} < \infty\},
	\end{aligned}
\end{equation}
where $ S'(\R^n)$ is the set of all tempered distributions on $\R^n$, $F$ and $F^{-1}$ are, respectively, the Fourier and inverse Fourier transforms. The family $\{\phi_j\}_{j\in \mathbb{N}}$ is composed of smooth compactly supported functions such that $\sum_{j=1}^\infty \phi_j =1$ designed to capture the behavior of the function $f$ on different frequency stripes. For the precise assumptions on the cut-off functions $\{\phi_j\}_{j\in \mathbb{N}}$, we refer the reader to \cite[Section 2.3.1]{triebel1983theory}. 

The Besov and Triebel-Lizorkin spaces on a bounded domain $\Omega$ are then defined as the set of restrictions
\begin{equation}
	\begin{aligned}
		&B^s_{pq}(\Omega) = \{f \,|\, \exists g \in B^s_{pq}(\R^n) \textrm{ with } \, g|_{\Omega} =f\},\\
		&F^s_{pq}(\Omega) = \{f \,|\, \exists g \in F^s_{pq}(\R^n) \textrm{ with } \, g|_{\Omega} =f\},
	\end{aligned}
\end{equation}
see, e.g., \cite[Section 3.2]{triebel1983theory}. The question of inner description (and related question of existence of extension operators~\cite{farkas2001sobolev}) of the spaces $B^s_{pq}(\Omega)$ and $F^s_{pq}(\Omega)$ is a delicate one for the general case, in particular for $0<p
\leq 1$. Nevertheless, for the case which is relevant in this work $(1<p<\infty)$, extension operators are available for $B^s_{pq}(\Omega)$ and $F^s_{pq}(\Omega)$, see, e.g.,~\cite{devore1993besov,seeger1989note}, as long as the domain $\Omega$ is of $(\varepsilon,\delta)$ type~\cite{jones1981quasiconformal}. This is in particular the case if $\Omega$ is Lipschitz~\cite[p. 73]{jones1981quasiconformal}. 

To simplify the presentation below, we introduce some short-hand notations. We denote the space of solutions for second-order parabolic equations by
\begin{equation}
	L_{p,q}^2(0,T) := L^p(0,T; W^{2,q}(\Omega)) \cap W^{1,p}(0,T;L^q(\Omega)).
\end{equation}
Similarly, the space
\begin{equation}
	L_{\infty,q}^1(0,T) := L^\infty(0,T; W^{1,q}(\Omega)) \cap W^{1,\infty}(0,T;L^q(\Omega)),
\end{equation}
will be the target solution space for the mass density $\varrho$. 

\section{The system of equations}\label{Sec:sys_eq}
On a bounded domain $\Omega \subset \R^3$, we consider the system of equations of NSF-MHD, which couples magnetic induction, see, e.g.,~\cite[Eqs. (2.2) \& (2.11)]{weiss2014magnetoconvection} to the Navier--Stokes--Fourier system: 
\begin{equation}\label{sys:raw}
	\begin{aligned}
		\pdt \varrho +\div_x(\varrho\bu) &= 0, \\
		\pdt (\varrho \bu) + \div_x (\varrho\bu\otimes\bu) + \nabla_x p(\varrho, \theta) &= \div_x \SS(\bu) + \varrho \nabla_x G + \curl_x \bB \times \bB ,\\
		\pdt \bB +\curl_x (\bB \times \bu)  + \curl_x(\xi  \curl_x \bB) &= 0,\\
		\div_x\bB &=0,\\
		\partial_t(\varrho e(\varrho,\theta)) + \div_x (\varrho e(\varrho,\theta) \bu) +\div_x \bq   & =  \mathbb{S} : \mathbb D_x \bu + \xi |\curl_x \bB|^2- p (\varrho,\theta) \div_x \bu,
	\end{aligned}
\end{equation}
with 
\begin{equation}\label{eq::viscosity_thing}
	\SS(\bu) =  \mu (2 \mathbb{D}_x\bu - \frac2d \div_x \bu \mathbb I)+ \lambda \div_x \bu \mathbb I,
\end{equation}
where  $\mu>0$ is the shear viscosity coefficient, $\lambda\geq 0$, the bulk viscosity coefficient, $\mathbb{D}_x \bu =\frac12 \nabla_x \bu +\frac12 \nabla_x\bu^t$ is the symmetrized gradient, and $\mathbb{I}$, the identity matrix in $\R^3$. 
Above, $\varrho$ is the mass density, $\bu$, the velocity, $p$, the pressure, $\theta$, the temperature, and $\bB$, the magnetic field. The coefficient $\xi$ is the magnetic resistivity and is a medium-dependent constant.

We further assume the Boyle-Mariotte equation of state which relates the pressure $p$ to the density $\varrho$ and temperature $\theta$:
\begin{equation}\label{eos:Boyle_Mariotte}
	p (\varrho,\theta) = \varrho \theta, \quad e = c_v \theta, \quad c_v>0,
\end{equation}
and the Fourier law of heat conduction:
\begin{equation}
	\bq = -\kappa \nabla_x \theta,
\end{equation}
where $\kappa>0$ is the heat conductivity of the medium, assumed to be constant, and $\bq$ a heat flux resulting from the temperature gradient $\nabla_x \theta$.

On the velocity field we impose the inhomogeneous Dirichlet boundary conditions:
\begin{equation}
	\begin{aligned}\label{BC:velocity}
		\bu|_{\partial \Omega} = \bu_B  \quad \textrm{on }[0,T]\times \partial\Omega.
	\end{aligned}
\end{equation}
We define the inflow, outflow, and wall subsets of $[0,T]\times \partial\Omega$ as
\begin{equation}
	\begin{aligned}
		\Sigma_{\textup{in}} & = \{(t,x) \in [0,T]\times \partial\Omega | \bu_B (t,x) \cdot n<0\},\\
		\Sigma_{\textup{out}} & = \{(t,x) \in [0,T]\times \partial\Omega | \bu_B (t,x) \cdot n>0\},\\
		\Sigma_{\textup{wall}} & = \{(t,x) \in [0,T]\times \partial\Omega | \bu_B (t,x) \cdot n=0\}.
	\end{aligned}
\end{equation}
To simplify the presentation, we assume that $\Sigma_{\textup{in}} = [0,T] \times \Gamma_{\textup{in}}$, i.e, we assume that the inflow boundary $\Gamma_{\textup{in}}$ is time-independent.

As is well-known, we need to prescribe the mass density on the inflow part of the boundary
\begin{equation}\label{BC:density}
	\varrho = \varrho_B>0 \quad \textup{ on } [0,T]\times\Gamma_{\textup{in}}.
\end{equation}	
We also prescribe Dirichlet boundary conditions for the magnetic and temperature fields:
\begin{equation}\label{BC:magneticfield}
	\begin{aligned}
		\bB\times n = b_1 \quad \textup{ on } [0,T]\times\partial \Omega, \\
		\theta = \theta_B>0 \quad \textup{ on } [0,T]\times\partial \Omega.
		\end{aligned}
\end{equation}
We additionally impose initial data 
\begin{equation}\label{Initial_data}
	\begin{aligned}
		\varrho(0) & =\varrho_0,\\
		\bu(0) &= \bu_0,\\
		\theta(0) & =\theta_0,\\
		\bB(0) &= \bB_0
	\end{aligned}
\end{equation}
to be chosen in suitable function spaces which will be made precise later in the analysis.

Using the EOS~\eqref{eos:Boyle_Mariotte}, we reformulate the system of equations~\eqref{sys:raw} as a closed system for the quantities $\varrho$, $\bu$, $\bB$, and  $\theta$:
\begin{equation}\label{eq:energy_equation_theta}
	\begin{aligned}
		&\pdt \varrho + \bu \cdot \nabla_x \varrho = -\varrho \div_x \bu,
		\\
		&\pdt \bu + \bu\cdot \nabla_x\bu =  \frac1{\varrho} [\div_x (\SS(\bu))+\curl_x\bB \times \bB]-\nabla_x \theta - \theta \nabla_x \log(\varrho) + \nabla_x G,
		\\
		& \pdt \bB   + \curl_x(\xi  \curl_x \bB)= \curl_x( \bu\times \bB), 
		\\
		& \div_x \bB = 0,
		\\
		&\pdt \theta + \bu\cdot \nabla_x\theta - \frac{\kappa}{\varrho c_v} \Delta_x \theta = \frac{1}{\varrho c_v} \SS(\bu):\mathbb{D}_x \bu + \frac{\xi}{\varrho c_v} |\curl_x \bB|^2 - \frac{\theta}{c_v} \div_x \bu .
	\end{aligned}
\end{equation}

\section{Main result: Theorem~\ref{theorem:main_thm}}\label{Sec:Main_Result}
The main result of this work is to show the local well-posedness of the open NSF-MHD system in the $L^p$--$L^q$ class, i.e., in the case where we allow an inflow boundary $\Gamma_{\textup{in}} \neq \emptyset$.

We begin by presenting the assumptions on the domain and initial and boundary data. The assumption on the inflow part of the boundary is justified by the discussion preceding Lemma~\ref{lemma::est_der}.
\begin{assumption}[Assumptions on the domain]\label{assu:domain}
	Let $\Omega$ be a bounded domain of $\R^d$ ($d=3$) with boundary $\partial \Omega$ of class at least $C^2$. We further assume that the inflow part of the boundary $\Gamma_{\textup{in}}$ is a closed surface.
\end{assumption}

\begin{remark}[On the domain regularity]
	The domain regularity assumption satisfies a number of requirements needed for the analysis. First, we have sufficient regularity to have inner descriptions and extension operators for the Sobolev, Besov and Triebel--Lizorkin spaces involved in the analysis (see discussion in Section~\ref{sec:notation}). Second, we ensure the needed regularity to apply the parabolic regularity estimates, which is a key ingredient in the analysis~\cite[Theorem 2.3]{denk2007optimal}.  
\end{remark}

\begin{assumption}[Assumptions on the data]\label{assu:data}

	Let $p,q \in (1,\infty)$. Let $T>0$. We assume the initial and boundary data
 	for the density satisfy
	\begin{equation}\label{data_density}
		\begin{aligned}
			&\varrho_B>0 \in W^{1,q}(0,T; L^q(\Gamma_\textup{in})) \cap L^{q}(0,T; W^{1, q}(\Gamma_\textup{in})),\\
			&\varrho_0>0 \in W^{1,q}(\Omega),\\
			& \varrho_0(x) = \varrho_B(x,0)  \qquad \ \textrm{for a.e.}\ x \in \Gamma_{\textup{in}},
		\end{aligned}
	\end{equation}
	for the velocity filed,
		\begin{equation}\label{data_velocity}
		\begin{aligned}
			&\bu_0 \in B^{2(1-1/p)}_{qp}(\Omega),\\
			&\bu_B \in F^{1-1/2q}_{pq}(0,T;L^q(\partial\Omega)) \cap L^p (0,T; W^{2-1/q,q}(\partial \Omega)),\\
			& \bu_0(x) = \bu_B(x,0)  \qquad \ \textrm{for a.e.}\ x \in \partial \Omega,\\
			& - \bu_B \cdot n\geq c >0   \qquad \ \textrm{for some }\ c>0 \ \textrm{ on } [0,T]\times \Gamma_{\textup{in}},
		\end{aligned}
	\end{equation}
	for the temperature field,
	\begin{equation}\label{data_temp}
		\begin{aligned}
			&\theta_0 \in B^{2(1-1/p)}_{qp}(\Omega),\\
			&\theta_B \in F^{1-1/2q}_{pq}(0,T;L^q(\partial\Omega)) \cap L^p (0,T; W^{2-1/q,q}(\partial\Omega)),\\
			& \theta_0(x) = \theta_B(x,0) \quad \textrm{for a.e.}\quad x \in \partial\Omega,
		\end{aligned}
	\end{equation}
	and for magnetic field,
	 \begin{equation}\label{data_mag}
		\begin{aligned}
			&\bB_0 \in B^{2(1-1/p)}_{qp}(\Omega),\\
			& b_1 \in F^{1-1/2q}_{pq}(0,T;L^q(\partial\Omega)) \cap L^p (0,T; W^{2-1/q,q}(\partial\Omega)),\\
			& \bB_0(x)\times n = b_1(x,0) \quad \textrm{for a.e.}\quad x \in \partial\Omega,
			& \div_x\bB_0 = 0 \quad \textrm{a.e. on}\quad x \in \partial\Omega.
		\end{aligned}
	\end{equation}
	Additionally, we assume that the gravity field $G \in L^p(0,T; W^{1,q}(\Omega))$.
\end{assumption}

Having introduced the assumptions made on the domain $\Omega$ and initial-boundary data, we are ready to present the main theorem of this work:
\begin{framed}
\begin{theorem}~\label{theorem:main_thm}
		Let $p,q \in (1,\infty)$, such that $q>d$ and $\max\{\frac{2q}{q-1}, \frac{2q}{2q-d}\}<p$. Let Assumptions~\ref{assu:domain} and \ref{assu:data} hold.
		
		 Then, there exists a strictly positive final time $T_*>0$ such that the system \eqref{eq:energy_equation_theta} supplemented with the boundary conditions~\eqref{BC:velocity}--\eqref{BC:magneticfield} and the initial conditions~\eqref{Initial_data} has a unique solution:
		\[(\varrho, \bu ,\theta,\bB) \in L^1_{\infty,q}(0,T_*) \times L^2_{p,q}(0,T_*) \times L^2_{p,q}(0,T_*) \times L^2_{p,q}(0,T_*). \]
\end{theorem}
\end{framed}

The proof of Theorem~\ref{theorem:main_thm} will be given in multiple steps and is detailed in Sections~\ref{Sec:Linearized} and \ref{Sec:Main_proof}. We present, in the next section, the minimum and maximum principles for the density and temperature.

\section{Minimum and maximum principles}\label{Sec:MinMax}

\subsection{Minimum-maximum principle for the density}
Provided the velocity field $\bu$ is regular enough to allow for the Lagrangian formulation of the equation of continuity, we infer from integrating over the streamlines (see, e.g., \cite[p. 16]{feireisl2022mathematics}) that the density will verify a minimum-maximum principle of the form
\begin{equation}\label{ineq:max_min}
	\begin{aligned}
\min\{\min_{\Omega}\varrho_0, \min_{\Gamma_{\textup{in}}}\varrho_B\} \exp&\left(-\int_0^\tau \|\div_x \bu\|_{L^\infty(\Omega)} \dt\right) \leq \varrho(\tau,\cdot) \\
&\leq \max\{\max_{\Omega}\varrho_0, \max_{\Gamma_{\textup{in}}}\varrho_B\} \exp\left(\int_0^\tau \|\div_x \bu\|_{L^\infty(\Omega)} \dt\right).
	\end{aligned}
\end{equation}
As mentioned in \cite{abbatiello2024local}, it is sufficient for the velocity to satisfy
\begin{equation}
	\bu \in L^1(0,T; W^{1,\infty}(\Omega)),
\end{equation}
for the aforementioned Lagrangian formulation, and subsequently \eqref{ineq:max_min}, to be valid.

\subsection{Minimum principle for the temperature} 
We multiply the energy equation \eqref{eq:energy_equation_theta} by \[\exp\left(\frac1{c_v}\int_0^\tau \|\div_x \bu\|_{L^\infty(\Omega)}\dt\right),\] then similarly to \cite{abbatiello2024local} we have:
\begin{equation}\label{eq:Theta_differential_ineq}
	\pdt \Theta + \bu\cdot\nabla_x \Theta - \frac{\kappa}{\varrho c_v} \Delta_x \Theta \geq 0,
\end{equation}
on $(0,T)\times \Omega$, with 
\[\Theta = \theta \exp\left(\frac 1 {c_v} \int_0^\tau \|\div_x \bu\|_{L^\infty(\Omega)}\dt\right).\]
We note that the boundary and initial data are then given by:
\begin{equation}
	\begin{aligned}
	\Theta|_{\partial \Omega} &= \Theta_B = \theta_B \exp\left(\frac 1 {c_v} \int_0^\tau \|\div_x \bu\|_{L^\infty(\Omega)}\dt\right),\\
	\Theta(0,\cdot) &= \Theta_0 = \theta_0. 
	\end{aligned}
\end{equation}

Since the differential inequality above involves a parabolic operator, we can apply the maximum principle \cite[Ch. 3, Theorem 5]{protter2012maximum} (to $- \Theta$) to find that $\Theta$ attains its minimum either at the boundary or at $t=0$. 
Thus \[\Theta \geq \min\{\Theta_0, \Theta_B\}, \]
which implies 
\begin{equation}
	\begin{aligned}
		\theta(\tau,\cdot) \geq \min \Bigg\{& \min_{\Omega} \theta_0 \exp\left(-\frac 1 {c_v} \int_0^\tau \|\div \bu\|_{L^\infty(\Omega)}\dt\right),
		& \min_{\Gamma_D, s\in(0,\tau)} \theta_B(s) \exp\left(-\frac 1 {c_v} \int_{s}^\tau \|\div \bu\|_{L^\infty(\Omega)}\dt\right)  \Bigg\}
	\end{aligned} 
\end{equation}
for all $\tau>0$.

\section{$L^p\--L^q$ setting: Linear Estimates}\label{Sec:Linearized}
In this section, we plan to use the $L^p\--L^q$ theory for parabolic equations~\cite{denk2007optimal} to provide estimates for the MHD-NSF model under study. Our approach will then use a Banach fixed-point argument to control the nonlinear terms in the system~\eqref{eq:energy_equation_theta} and will be in spirit close to, e.g., the one described in~\cite[Chapter 5]{feireisl2022mathematics}
\begin{enumerate}[i)]
	\item fix some velocity field $\bv$ which verifies the boundary and initial condition of $\bu$;
	\item solve the continuity equation for $\varrho$ with fixed velocity $\bv$;
	\item solve the equation for the temperature $\theta$ with fixed velocity and magnetic fields $\bv$ and $\bB$;
	\item solve for the magnetic field with $\bu = \bv$;
	\item solve for $\bu$ with the drift term linearized to $\bv\cdot \nabla_x \bu$;
	\item apply a fixed point to find $(\varrho, \bu, \theta, \bB)$ which solves the nonlinear problem~\eqref{eq:energy_equation_theta} with the boundary data~\eqref{BC:velocity}, \eqref{BC:density}, \eqref{BC:magneticfield} and initial data~\eqref{Initial_data}.
\end{enumerate}

We write the linearized system of equation used for the fixed-point mapping definition
\begin{equation}\label{sys::linearized_system}
	\begin{aligned}
		&\pdt \varrho + \bv \cdot \nabla_x \varrho = -\varrho \div_x \bv,
		\\
		&\pdt \bu + \bv\cdot \nabla_x\bu =  \frac1{\varrho} [\div_x (\SS(\bu))+\curl_x\bB \times \bB]-\nabla_x \theta - \theta \nabla_x \log(\varrho) + \nabla_x G,
		\\
		& \pdt \bB + \curl_x(\xi  \curl_x \bB)= \curl_x (\bv \times \bB),\\
		& \div_x \bB = 0,
		\\
		&\pdt \theta + \bv\cdot \nabla_x\theta - \frac{\kappa}{\varrho c_v} \Delta \theta = \frac{1}{\varrho c_v} \SS(\bu):\mathbb{D}_x \bv + \frac{\xi}{\varrho c_v} |\curl_x \bB|^2 - \frac{\theta}{c_v} \div_x \bv .
	\end{aligned}
\end{equation}

\subsection{Estimate for the continuity equation}
In what follows, we establish, following the approach of \cite[Lemma 4.1]{kotschote2012strong}, an initial estimate for the continuity equation with a right-hand side of the form
\begin{equation}\label{eq:rho}
	\begin{aligned}
	\pdt \varrho + \bv \cdot \nabla_x \varrho + \varrho \div_x \bv = f &\qquad \textrm{on } [0,T] \times \Omega,\\
	\varrho(0,\cdot) = \varrho_0  &\qquad \textrm{on } \Omega,\\
	\varrho(\cdot,x) = \varrho_B &\qquad \textrm{on } \Gamma_{\textup{in}},
	\end{aligned} 	   
\end{equation}
where $\varrho_0$ and $\varrho_B$ are prescribed initial data with regularity clarified below, and the source term $f$, independent of $\varrho$, is subject to the condition:
\begin{equation}\label{ineq:f_Lp}
\|f(t)\|_{L^p(\Omega)} \leq  k_1(t), \qquad k_1(t) \,\in L_1(0,T),
\end{equation}
for almost every $t\in(0,T)$.
The main difference with the aforementioned work of Kotschote~\cite{kotschote2012strong} is that we do not rely on the condition that $\bu_B\cdot n \geq 0 $ (no incoming flow). Rather, we use the fact that we know the mass density $\varrho_B$ at the inflow part of the boundary.

\begin{remark}[On the source term $f$]
	The source term $f$ is artificially added for the purpose of being used for the contraction estimate for the mass density; see Lemma~\ref{lemma:contraction_dens}. For the original system~\eqref{eq:energy_equation_theta}, we point out that $f=0$. Therefore, except for Lemma~\ref{lemma::est_continuity}, we will set $f=0$.
\end{remark}

\begin{lemma}\label{lemma::est_continuity}
	Let $\Omega$ be a bounded domain in $\R^3$ and $p \in (1,\infty)$. Further, assume that $\varrho_0 \in L^p(\Omega)$, $\varrho_B \in L^p(0,T;L^p(\partial \Omega, [\bu_B\cdot n]^- \dsig)$, $\bv \in L^1(0,T; W^{1,\infty}(\Omega))$, and $f$ fulfills \eqref{ineq:f_Lp}. Then a solution $\varrho$ to \eqref{eq:rho} satisfies the estimate:
	\begin{equation}\label{ineq:Lp_rho}
		\begin{aligned}
		\|&\varrho\|_{L^\infty(0,T;L^p(\Omega))} \\ &\leq \Bigg\{ \left(\|\varrho_0\|_{L^p(\Omega)} +\|\varrho_B\|_{L^p(0,T;L^p(\partial \Omega, [\bu_B\cdot n]^- \dsig)}+ \|k_1\|_{L^1(0,T)}\right) \exp \left[  \int_0^T[\|\div_x \bv\|_{L^\infty(\Omega)} ] \dt\right] \Bigg\}.
		\end{aligned}
	\end{equation}
Assume, moreover that $\varrho_0\in L^\infty(\Omega)$, $\varrho_B \in L^\infty(0,T;L^\infty(\partial \Omega, [\bu_B\cdot n]^- \dsig)$ and that $k_1$ in \eqref{ineq:f_Lp} is independent of $p$. Then, we conclude that
	\begin{equation}\label{ineq:Linf_rho}
		\begin{aligned}
			\|&\varrho\|_{L^\infty(0,T;L^\infty(\Omega))} \\ &\leq \Bigg\{ \left(\|\varrho_0\|_{L^\infty(\Omega)} +\|\varrho_B\|_{L^\infty(0,T;L^\infty(\partial \Omega, [\bu_B\cdot n]^- \dsig)}+ \|k_1\|_{L^1(0,T)}\right) \exp \left[  \int_0^T[\|\div_x \bv\|_{L^\infty(\Omega)} ] \dt\right] \Bigg\}.
		\end{aligned}
	\end{equation}
 \end{lemma}
For a real-valued function $v: A \to \R^+$, the notation $[\cdot]^-$, above should be interpreted as 
\[[v]^-(t) = \begin{cases}
	v(t) \qquad\textrm{if }  v(t) \leq 0, \\
	0 	\quad\qquad\textrm{if }  v(t) > 0,
\end{cases}\]
for $t\in A$.
\begin{proof}
	Inspired by the proof of \cite[Lemma 4.2]{kotschote2012strong}, we multiply \eqref{eq:rho} by $p |\varrho|^{p-2} \varrho$, integrating over $\Omega$ by parts, we obtain  
\begin{equation}
		\begin{aligned}
		\ddt  \|\varrho\|_{L^p(\Omega)}^p =& (1-p) \int_{\Omega} \div_x \bv |\varrho|^p \dx + p \int_\Omega f  |\varrho|^{p-2} \varrho \dx - \int_{\partial\Omega} |\varrho|^p \bu_B\cdot n \dsig
		\\ &\begin{multlined}
		\leq p\,  \|\div_x \bv\|_{L^\infty(\Omega)}  \|\varrho\|^p_{L^p(\Omega)} \\ + p\, k_1(t) \|\varrho (t)\|_{L^p(\Omega)}^{p-1} - \int_{\partial \Omega} |\varrho_B|^p [\bu_B\cdot n]^- \dsig,
	\end{multlined}
	\end{aligned}
\end{equation}
where in the last line, we have used the boundary condition on the inflow of the boundary and assumption~\eqref{ineq:f_Lp}. 

Thus integrating in time over $[0,T]$, and applying the Gronwall--Perov inequality~\cite[p.~360]{mitrinovic1991inequalities}, yields:
\begin{equation}
	\begin{multlined}
		  \|\varrho\|_{L^p(\Omega)}^p \leq \Bigg\{ \left( \|\varrho_0\|_{L^p(\Omega)} +\|\varrho_B\|_{L^p(0,T;L^p(\partial \Omega, [\bu_B\cdot n]^- \dsig)}\right)   \exp \left[  \int_0^T\|\div_x \bv\|_{L^\infty(\Omega)}  \dt\right] \\+   \int_0^T k_1(t) \exp \left[  \int_t^T\|\div_x \bv(r)\|_{L^\infty(\Omega)} \,\textup{d}r\right] \dt 
		 \Bigg\}^{p},
	\end{multlined}
\end{equation}
thus,
\begin{equation}
	\begin{multlined}
 \|\varrho\|_{L^p(\Omega)} \leq \Bigg\{ \left(\|\varrho_0\|_{L^p(\Omega)} +\|\varrho_B\|_{L^p(0,T;L^p(\partial\Omega, [\bu_B\cdot n]^- \dsig)}+ \|k_1\|_{L^1(0,T)}\right)\\  \exp \left[  \int_0^T\|\div_x \bv\|_{L^\infty(\Omega)}   \dt\right] \Bigg\},
\end{multlined}
\end{equation}
which shows \eqref{ineq:Lp_rho}. Letting $p\to \infty$ in \eqref{ineq:Lp_rho} yields \eqref{ineq:Linf_rho}.
\end{proof}

\subsubsection{Existence and uniqueness of $\varrho$}
The existence of a unique weak solution follows from the method of characteristics. Indeed, we assume that we have enough regularity on the velocity field $\bv\in L^1(0,T; W^{1,\infty} (\Omega))$, for the characteristics to be well-defined.

For $\bv\in L^1(0,T; W^{1,\infty} (\Omega))$ fixed, it is immediate to see that the solution $\varrho$ is unique. Indeed, assume $\varrho_1$, $\varrho_2$ are two solutions to~\eqref{eq:rho} and consider the equation verified by the difference $\varrho_1-\varrho_2$, which corresponds to the homogeneous problem \eqref{eq:rho} (with $f =0$, $\varrho_0 = 0$, $\varrho_B = 0$), then by estimate~\eqref{ineq:Lp_rho}, $\varrho_1-\varrho_2 =0$.

\subsubsection{Higher regularity in time}
We can expect higher regularity in time by testing by $|\varrho_t|^{p-2}\varrho_t$, thus obtaining control on $\|\varrho_t\|_{L^p(0,T;L^p(\Omega))}$. This in turn yields control on the term $\div_x(\varrho \bu)$ in $L^p(0,T;L^p(\Omega))$. 
\subsubsection{Higher-order estimates in space of the density $\varrho$}

In order to show the existence of regular enough (strong) solutions of the continuity equation, we need to assume additional assumptions compared to Lemma~\ref{lemma::est_continuity}. In particular, we need to estimate the $L^p$ norms of the traces of the derivatives at the boundary. This forces us to assume regular data on the boundary and a uniform boundedness from below of $-\bu_B\cdot n \geq c>0 $ on $\Gamma_{\textup{in}}$. This is consistent with observations made in \cite{fiszdon1983initial,valli1982existence}. This requirement, together with the assumption that the velocity field is at least in $W^{1,\infty}(\Omega)$ means that $\Gamma_{\textup{in}}$ is disconnected from the rest of the boundary. Thus, similarly to \cite[Section B]{valli1986navier}, we require that $\Gamma_{\textup{in}}$ is a closed surface.
\begin{lemma}[Higher-order spatial estimates]\label{lemma::est_der}
	Let $p,q \in (1,\infty)$, such that $q>d$, $2 (1-1/p) > \frac {d}{q}$ (equivalently $\frac{2q}{2q-d}<p$), and 
		\begin{equation}\label{cond_p_q}
		1-\frac2p + \frac 1q  \geq 0.
		\end{equation}
		 Let the velocity field $\bv \in L^p(0,T; W^{2,q}(\Omega)) \cap W^{1,p}(0,T; L^q(\Omega))$%, the source term $f\in L^1(0,T;W^{1,q}(\Omega)) $
		  and let $\varrho_0$, $\varrho_B$ verify~\eqref{data_density}.
	  Then, there exists a unique solution  $\varrho \in L^\infty(0,T;W^{1,q}(\Omega))$  of \eqref{eq:rho} (with source term $f=0$) and it verifies the following estimate:
		 \begin{equation}\label{estimate_derivative_density}
		 	\begin{aligned}
		 	|\varrho &|_{L^\infty(0,T; W^{1,q}(\Omega))}^q\\ \lesssim & \exp \left(\|\div_x \bv\|_{L^1(0,T;L^\infty(\Omega))} + \int_0^T|\bv|_{W^{1,\infty}(\Omega)} \dt + \int_0^T |\bv|_{W^{2,q}(\Omega)}\|\varrho\|_{L^\infty(\Omega)} \dt \right)% \\ & \|f\|_{L^1(0,T;W^{1,q}(\Omega))} \Big)
		 	 \\&\Bigg\{|\varrho_0 |_{W^{1,q}(\Omega)}^q  +  \| \bv \cdot n\|_{L^\infty(0,T;L^\infty(\Gamma_{\textup{in}}))} \|\varrho_B\|_{L^q(0,T;L^q(\Gamma_{\textup{in}}))}^q\\&+  \|\pdt \varrho_B\|_{L^q(0,t;L^q(\Gamma_{\textup{in}}))}^q + \|\bv\|^q_{L^\infty(0,t;L^\infty(\Gamma_{\textup{in}}))} \|\varrho_B\|_{L^q(0,t;W^{1,q}(\Gamma_{\textup{in}}))}^q \\ &+\|\div_x \bv  \|_{L^q(0,t;L^q(\Gamma_{\textup{in}}))}^q\|\varrho_B\|^q_{L^\infty(0,t;L^\infty(\Gamma_{\textup{in}}))}, 	 \Bigg\}
		 	\end{aligned}
		 \end{equation}
	 where $|\cdot|_{W^{k,p}(\Omega)}$ stands for the semi-norm of $W^{k,p}(\Omega)$ and the hidden constant is independent of the final time $T$.
\end{lemma}

Before going into the proof let us point out that the assumption $$\bv \in L^p(0,T; W^{2,q}(\Omega)) \cap W^{1,p}(0,T; L^q(\Omega))$$ matches the regularity we expect to obtain for the velocity $\bu$ using the standard $L^p$--$L^q$ framework of, e.g., \cite{denk2007optimal}. On the other hand, according to \cite[Chapter III, Theorem 4.10.2]{amann1995linear}, 
\begin{equation}\label{reg_trace_v}
	\bv \in C([0,T]; B^{2(1-1/p)}_{qp} (\Omega)),
\end{equation}
where $B^{2(1-1/p)}_{qp} (\Omega)  = (L^q(\Omega), W^{2,q}(\Omega))_{1-1/p,p}$ is a Besov space and matches the interpolation of the Sobolev spaces to which $\bv$ and its time derivative belong; see~\cite[Theorem 3.3.6]{triebel1983theory} (see also \cite[Section 3.4.2]{triebel1983theory} for interpretation of Sobolev spaces as a special case of a Triebel--Lizorkin space). We note also that the embedding 
\begin{equation}\label{embedding_Besov_cont}
	B^{2(1-1/p)}_{qp} (\Omega) \hookrightarrow C(\overline{\Omega})
\end{equation}
holds on account of the condition  $2 (1-1/p) > \frac {d}{q}$, see~\cite[Theorem 3.3.1]{triebel1983theory}.

\begin{proof}
	First note that the condition $q > d$, ensures the embedding  \[L^p(0,T; W^{2,q}(\Omega)) \hookrightarrow L^1(0,T; W^{1,\infty}(\Omega))\] needed for the method of characteristics to be valid.
	
	To prove the estimate, we use the approach of \cite[Lemma 2.4]{valli1982existence}. Let $\alpha$ be a multi-index with $|\alpha|\leq 1$. Then, taking applying $D^\alpha$ to the continuity equation \eqref{eq:rho} (with source term $f=0$) and testing with $ q |D^\alpha \varrho|^{q-2} D^\alpha \varrho$, one obtains 
	\begin{equation}\label{eq:identity_rho_der}
		\begin{aligned}
		\ddt \|D^\alpha\varrho\|_{L^q(\Omega)}^q = & - q \int_\Omega D^\alpha (\bv \cdot \nabla_x\varrho) |D^\alpha\varrho|^{q-2} D^\alpha\varrho  \dx - q \int_{\Omega} D^\alpha(\varrho \div_x \bv) |D^\alpha\varrho|^{q-2} D^\alpha\varrho \dx. 
		\end{aligned} 
	\end{equation}
	As identified by Valli~\cite[Lemma 2.4]{valli1982existence}, the following term requires integration by parts to be controlled
	\[- \int_{\Omega} \bv\cdot(\nabla_x(|D^\alpha \varrho|^q)) = \int_{\Omega} \div_x \bv \, (|D^\alpha \varrho|^q) \dx - \int_{\partial\Omega} \bv \cdot n (|D^\alpha \varrho|^q) \dsig.\] 

\subsection*{Controlling the boundary term.} 
Compared to the case $\alpha=0,$ we do not have \emph{a priori} information on the trace of $D^\alpha \varrho$ on $\Gamma_\textup{in}$, therefore we need to establish such an estimate. To this end, following the idea of \cite{fiszdon1983initial}, we use the natural coordinate system near the boundary formed of the orthonormal vectors $n$ (outward normal) and $\{\tau_i\}_{1 \leq i \leq d-1}$ (tangential unit vectors). 

In the tangential directions, the control is directly given by the regularity of the boundary condition $\varrho_B$ and that of the velocity field $\bv$:
\begin{equation}\label{ineq:control_tan}
	\begin{aligned}
		- \int^t_0 \int_{\Gamma_{\textup{in}}} \bv \cdot n \left|\frac{\partial}{\partial \tau_i}\varrho\right|^q \dsig \dt   = &- \int^t_0 \int_{\Gamma_{\textup{in}}} \bv \cdot n \left|\frac{\partial}{\partial \tau_i}\varrho_B\right|^q \dsig \dt \\ \leq & \| \bv \cdot n\|_{L^\infty(0,T;L^\infty(\Gamma_{\textup{in}}))} \|\varrho_B\|_{L^q(0,T;W^{1,q}(\Gamma_{\textup{in}}))}^q,
	\end{aligned}
\end{equation}
where $i\in \{1, \ldots ,d-1\}.$

To establish an estimate in the normal direction to the boundary, we lift the boundary condition $\varrho_B$ to a function on $\Omega$ (which we still denote $\varrho_B$).
Let $\overline{\varrho} = \varrho - \varrho_B$, then $\overline{\varrho}$, verifies:
\begin{equation}\label{eq:homogeneous_prob}
	\pdt \overline\varrho + \bv \cdot \nabla_x \overline\varrho + \div_x \bv \overline\varrho = -\pdt \varrho_B - \bv \cdot \nabla_x \varrho_B - \div_x \bv \varrho_B.
\end{equation}
Formally, evaluating \eqref{eq:homogeneous_prob} on $\Gamma_{\textup{in}}$, we obtain:
\begin{equation}\bv \cdot \nabla_x \overline\varrho = -\pdt \varrho_B - \div_x (\varrho_B \bv)
\end{equation}
owing to  $\overline\varrho |_{\Gamma_{\textup{in}}} = 0$. We further simplify by noticing that the tangential component of $\nabla \overline{\varrho} $ is equal to $0$, such that 
\[\bv \cdot \nabla_x \overline\varrho = (\bv \cdot n) (\nabla_x \overline\varrho \cdot n) = (\bv \cdot n) \frac{\partial}{\partial n} \overline\varrho,\]
this provides control on the normal derivatives as long as $\bv\cdot n$ is bounded away from $0$ on $\Gamma_{\textup{in}}$. Indeed, we can test the resulting equation on the boundary by $\left|\frac{\partial}{\partial n}\overline\varrho\right|^{q-2} \frac{\partial}{\partial n}\overline\varrho$ and integrate over the inflow part of the boundary ($\Gamma_{\textup{in}}$) and in time on $(0,t)$, such that  we obtain
\begin{equation}
	\begin{aligned}
		\int^t_0 \int_{\Gamma_{\textup{in}}} (- \bv \cdot n) \left|\frac{\partial}{\partial n}\overline\varrho\right|^q \dsig \dt =  & \int^t_0 \int_{\Gamma_{\textup{in}}} \left( \pdt \varrho_B + \varrho_B \div_x \bv + \bv \cdot \nabla_x \varrho_B \right) \left|\frac{\partial}{\partial n}\overline\varrho\right|^{q-2} \frac{\partial}{\partial n}\overline\varrho \dsig \dt\\
		\leq & C(\varepsilon) \int^t_0 \int_{\Gamma_{\textup{in}}} \left| \pdt \varrho_B + \varrho_B \div_x \bv + \bv \cdot \nabla_x \varrho_B \right|^q \dsig \dt+ \varepsilon \int^t_0 \int_{\Gamma_{\textup{in}}}  \left|\frac{\partial}{\partial n}\overline\varrho\right|^{q}  \dsig \dt,\\
	\end{aligned}  
\end{equation}
where the last line was obtained using Young's inequality with exponents $q$ and $\frac{q}{q-1}$. The constant $\varepsilon>0$ is arbitrary, $C(\varepsilon)>0$ is a $\varepsilon$-dependent constant. Now, we use the lower bound of the velocity on the inflow part of the boundary $c>0$ as expressed in \eqref{data_velocity}:
\begin{equation}
	\begin{aligned}
		c \int^t_0 \int_{\Gamma_{\textup{in}}} \left|\frac{\partial}{\partial n}\overline\varrho\right|^q \dsig \dt
		\leq & C(\varepsilon) \int^t_0 \int_{\Gamma_{\textup{in}}} \left| \pdt \varrho_B + \varrho_B \div_x \bv + \bv \cdot \nabla_x \varrho_B \right|^q \dsig \dt+ \varepsilon \int^t_0 \int_{\Gamma_{\textup{in}}}  \left|\frac{\partial}{\partial n}\overline\varrho\right|^{q}  \dsig \dt,\\
	\end{aligned}  
\end{equation}
thus choosing $\varepsilon$ smaller than the constant $c$ (e.g., $\frac c2$), we obtain that 
\begin{equation}\label{est_boundary_in}
	\begin{aligned}
		\left(\int^t_0 \int_{\Gamma_{\textup{in}}} \left|\frac{\partial}{\partial n}\overline\varrho\right|^q \dsig \dt\right)^{1/q} & \lesssim \|\pdt \varrho_B + \bv \cdot \nabla_x\varrho_B + \div_x \bv \varrho_B \|_{L^q(0,t;L^q(\Gamma_{\textup{in}}))} \\
		&\lesssim \|\pdt \varrho_B\|_{L^q(0,t;L^q(\Gamma_{\textup{in}}))} + \|\bv\|_{L^\infty(0,t;L^\infty(\Gamma_{\textup{in}}))} \|\varrho_B\|_{L^q(0,t;W^{1,q}(\Gamma_{\textup{in}}))} \\ 
		& \qquad + \|\div_x \bv \varrho_B \|_{L^q(0,t;L^q(\Gamma_{\textup{in}}))}.
	\end{aligned}
\end{equation}

\subsubsection*{Estimating $ \|\div_x \bv \varrho_B \|_{L^q(0,t;L^q(\Gamma_{\textup{in}}))}$.}
We present the estimate procedure for the more involved case $p<q$. Indeed, the estimate \eqref{eq:estimate_divv_rho} below holds in a more elementary way for the case $p\geq q$.

First, we identify the trace space of $\div_x \bv$ on $\Gamma_{\textup{in}}$ which are readily identified as trace spaces of Sobolev spaces. Then, owing to \cite[Chapter III, Theorem 4.10.2]{amann1995linear} (see also \cite[Corollary 3.12.3]{bergh2012interpolation}), we know that 
\begin{equation}\label{divx_embedding}
	\div_x \bv \in C([0,t]; B^{1-\frac2p -\frac1q}_{qp} (\Gamma_{\textup{in}})) \cap L^p(0,t; W^{1-\frac1q,q} (\Gamma_{\textup{in}})) \hookrightarrow L^q (0,t; B^{1-\frac2p + \frac 1q }_{qq'} (\Gamma_{\textup{in}})),
\end{equation}
and 
$$\varrho_B \in C([0,t]; B^{1-\frac 1q}_{qq} (\Gamma_{\textup{in}})) \cap L^{q}(0,T; W^{1, q}(\Gamma_\textup{in})),$$
where for the embedding in \eqref{divx_embedding} we used real interpolation with the parameters $0<p/q<1$ and $q'$, where $q'$ is the conjugate H\"older exponent of $q$. (Note that interpolation of $L^p$ spaces (time component) is simply obtained using H\"older's inequality).

We note that $q>d\geq 2 >q'$ and that $1-\frac2p + \frac 1q  \geq 0$ due to \eqref{cond_p_q}. Thus using standard properties of Besov spaces~\cite[Section 2.3.2, Propostition 2]{triebel1983theory} (see also the remark after \cite[Proposition 3.2.4]{triebel1983theory} and \cite[Proposition 3.3.1]{triebel1983theory}), we know the following embeddings hold:
$$B^{1-\frac2p + \frac 1q}_{qq'} (\Gamma_{\textup{in}}) \hookrightarrow B^{0}_{qq'}  (\Gamma_{\textup{in}}) \hookrightarrow F^{0}_{q2}  (\Gamma_{\textup{in}}) = L^q (\Gamma_{\textup{in}}).$$

On the other hand, we also have the embedding
$B^{1-\frac 1q }_{qq} (\Gamma_{\textup{in}}) \hookrightarrow C(\Gamma_{\textup{in}}) \hookrightarrow L^\infty(\Gamma_{\textup{in}}),$ due to the condition $q>d$.

Putting these embeddings together, we can estimate the last term in \eqref{est_boundary_in} as 
\begin{equation}\label{eq:estimate_divv_rho}
	\|\div_x \bv \varrho_B \|_{L^q(0,t;L^q(\Gamma_{\textup{in}}))}\lesssim  \|\div_x \bv  \|_{L^q(0,t;L^q(\Gamma_{\textup{in}}))}\|\varrho_B\|_{L^\infty(0,t;L^\infty(\Gamma_{\textup{in}}))}.
\end{equation}	

\subsection*{Establishing an estimate for $\varrho$}
Going back to \eqref{eq:identity_rho_der}, 
\begin{equation}
	\begin{aligned}
\ddt \|D^\alpha\varrho\|_{L^q(\Omega)}^q \leq &  q  \|D^\alpha \bv\|_{L^\infty(\Omega)} \|\nabla_x \varrho\|_{L^q(\Omega)} \|D^\alpha \varrho\|_{L^q(\Omega)}^{q-1}  - \int_{\partial\Omega} \bv \cdot n (|D^\alpha \varrho|^q) \dsig \\ &+ \|\div_x \bv\|_{L^\infty(\Omega)} \|D^\alpha\varrho\|_{L^q(\Omega)}^{q} \\ &
+  \|\bv\|_{W^{2,q}(\Omega)} \|D^\alpha\varrho\|_{L^q(\Omega)}^{q-1} \|\varrho\|_{L^\infty(\Omega)}. 
	\end{aligned}
\end{equation}
Thus summing over all multi-indices $\alpha$ with $|\alpha|\leq1$, and recalling that $|\cdot|_{W^{1,q}(\Omega)}$ denotes the semi-norm of $W^{1,q}(\Omega)$, we obtain after integration on $(0,T)$ 
\begin{equation}
	\begin{aligned}
		 |\varrho|_{W^{1,q}(\Omega)}^q \lesssim \,& |\varrho_0|_{W^{1,q}(\Omega)}^q + \int_0^T |\bv|_{W^{1,\infty}(\Omega)} |\varrho|_{W^{1,q}(\Omega)}^q \dt\\ & + \int_0^T \|\div_x \bv\|_{L^\infty(\Omega)} |\varrho|_{W^{1,q}(\Omega)}^{q} \dt \\ &
		+  \int_0^T\|\bv\|_{W^{2,q}(\Omega)} |\varrho|_{W^{1,q}(\Omega)}^{q-1} \|\varrho\|_{L^\infty(\Omega)} \dt \\& 	+ 
		\| \bv \cdot n\|_{L^\infty(0,T;L^\infty(\Gamma_{\textup{in}}))} \|\varrho_B\|_{L^q(0,T;L^q(\Gamma_{\textup{in}}))}^q\\&+  \|\pdt \varrho_B\|^q_{L^q(0,T;L^q(\Gamma_{\textup{in}}))} + \|\bv\|_{L^\infty(0,T;L^\infty(\Gamma_{\textup{in}}))}^q \|\varrho_B\|_{L^q(0,T;W^{1,q}(\Gamma_{\textup{in}}))}^q \\ &+\|\div_x \bv  \|^q_{L^q(0,T;L^q(\Gamma_{\textup{in}}))}\|\varrho_B\|^q_{L^\infty(0,T;L^\infty(\Gamma_{\textup{in}}))}. 		
	\end{aligned}
\end{equation}
We apply, similarly to the proof of Lemma~\ref{lemma::est_continuity}, Gronwall--Perov's inequality, which yield the desired estimate.

\end{proof}

Finally, it is worth mentioning that under the assumptions of Lemmas~\ref{lemma::est_continuity} and \ref{lemma::est_der}, the solution $\varrho$ constructed through the methods of characteristics is a strong solution. Indeed, given that
\[\pdt \varrho = - \bv \cdot \nabla_x \varrho  -\varrho \div_x \bv\] weakly, as a corollary of Lemmas~\ref{lemma::est_continuity} and \ref{lemma::est_der}, we can conclude that \[\pdt \varrho \in L^p(0,T; L^\infty(\Omega))  \cap L^\infty(0,T; L^q(\Omega)). \]
\subsection{Parabolic estimate for the velocity field }
When fixing the other quantities, the estimates for the velocity ($\bu$), temperature ($\theta$), and magnetic ($\bB$) fields are applications of the $L^p$--$L^q$ theory \cite[Theorem 2.3]{denk2007optimal} and are given in the Propositions~\ref{prop::estimate_lin_velocity_field}, \ref{prop::estimate_lin_temp_field}, and \ref{prop::linearized_mag} below. The idea is to combine these estimates with a fixed-point argument to show the well-posedness of the system~\eqref{eq:energy_equation_theta} supplemented with boundary data~\eqref{BC:velocity}--\eqref{BC:magneticfield} and initial data~\eqref{Initial_data}.

For the convenience of the reader, we recall here the linearized momentum equation:
\begin{equation}\label{eq::linearized_velocity}
	\begin{aligned}
	\pdt \bu + \bu_0 \cdot \nabla_x\bu - \frac1{\varrho_0} \div_x (\SS(\bu))=  & \, \left(\frac1{\varrho} - \frac1{\varrho_0}\right) \div_x (\SS(\bu)) + (\bu_0-\bv) \cdot \nabla_x \bu  \\ & \,+  \frac1{\varrho} \curl_x\bB \times \bB-\nabla_x \theta - \theta \nabla_x \log(\varrho) + \nabla_x G,\\
	\bu(0) = &\bu_0,\\
	\bu|_{\partial\Omega} =& \bu_B.
	\end{aligned}
\end{equation}
where $\bv$, $\theta$, $\bB$, and the gravity potential $G$ are fixed. The initial conditions $\bu_0$, $\varrho_0$ were introduced to avoid the dependence of the well-posedness constants on the variables of the fixed-point operator.
\begin{proposition}\label{prop::estimate_lin_velocity_field}
	Let $T>0$. Let $(p,q) \in (1,\infty)\times (d,\infty)$. Let $\varrho \in L^1_{\infty,q}(0,T)$ with $\varrho\geq r_0>0$ on $[0,T]\times \overline\Omega$ and $\varrho(0) = \varrho_0$ on $\Omega$. Let $\bB,\bv,\theta \in L^2_{p,q}(0,T)$, $G \in L^p(0,T;W^{1,q}(\Omega))$, and let the initial data and boundary data verify~\eqref{data_velocity}.
We further assume that $\bv (0) = \bu_0.$

Then there exists a unique $\bu \in  L^2_{p,q}(0,T) $ which solves \eqref{eq::linearized_velocity}. Furthermore, there exists a constant $\Lambda$ which may depend on $\varrho_0$, $\bu_0$, $p$, $q$ and $T$ such that 
\begin{equation}\label{est::prop_lin_velocity}
	\begin{aligned}
		\int_0^\tau \big[\| \pdt \bu \|^p_{L^q(\Omega)} + & \|\bu\|_{W^{2,q}(\Omega)}^p \big]\dt   \\ \leq & \Lambda\Bigg[ \left \|\frac{1}{\varrho} \curl_x\bB \times \bB \right \|_{L^p(0,\tau; L^q(\Omega))}^p + \|\nabla_x \theta \|^p_{L^p(0,\tau; L^q(\Omega))} \\ &+ \|\theta \nabla_x \log(\varrho) \|^p_{L^p(0,\tau; L^q(\Omega))} + \|G\|^p_{L^p(0,\tau;W^{1,q}(\Omega))}\\ &+  \|\bu_0\|_{B^{2(1-1/p)}_{qp}(\Omega)}^p + \| \bu_B\|^p_{F^{1-1/2q}_{pq}(0,\tau;L^q(\partial\Omega)) \cap L^p (0,\tau; W^{2-1/q,q}(\partial \Omega))}\Bigg],
	\end{aligned}
\end{equation}
holds for all $\tau\in [0,T_{\mathrm{eff},1}\wedge T]$, where $T_{\mathrm{eff},1}>0$ depends on $\Lambda$, $\|\frac1\varrho\|_{W^{1,p}(0,T;L^q(\Omega))}$, $\|\bv\|_{W^{1,p}(0,T;L^q(\Omega))}$, and various embedding constants.
\end{proposition}

\begin{proof}
	As announced before, we plan to use \cite[Theorem 2.3]{denk2007optimal}, therefore we need to verify its assumptions. First, the regularity assumption on the coefficients of the differential operator on the left-hand-side of \eqref{eq::linearized_velocity} are verified thanks to the embedding 
	$$ W^{1,q}(\Omega)\hookrightarrow C(\overline \Omega), \qquad \textrm{for } q>d,$$ 
	and the strict positivity of $\varrho_0\geq r_0>0$.
	The ellipticity of the differential operator is ensured through the positivity and nonnegativity assumption on the viscosity coefficients $\mu>0$ and $\lambda\geq0$ in the definition of the operator $\SS(\bu)$ in~\eqref{eq::viscosity_thing} together with $\frac1{\varrho_0}>0$. We thus obtain the estimate
	\begin{equation}\label{velocity_first_ineq}
		\begin{aligned}
			\int_0^T \big[\| \pdt \bu \|^p_{L^q(\Omega)} + & \|\bu\|_{W^{2,q}(\Omega)}^p \big]\dt   \\ \leq & \Lambda(\varrho_0,\bu_0, p, q,T)  \Bigg[ \int_0^T \left\|  \left(\frac1{\varrho} - \frac1{\varrho_0}\right) \div_x (\SS(\bu)) \right\|_{L^q(\Omega)}^p \dt  + \int_0^T \|(\bu_0 - \bv ) \cdot \nabla_x \bu\|^p_{L^q(\Omega)}  \dt\\ &+ \left \|\frac{1}{\varrho} \curl_x\bB \times \bB \right \|_{L^p(0,T; L^q(\Omega))}^p + \|\nabla_x \theta \|^p_{L^p(0,T; L^q(\Omega))} \\ &+ \|\theta \nabla_x \log(\varrho) \|^p_{L^p(0,T; L^q(\Omega))} + \|G\|^p_{L^p(0,T;W^{1,q}(\Omega))}\\ &+  \|\bu_0\|_{B^{2(1-1/p)}_{qp}(\Omega)}^p + \| \bu_B\|^p_{F^{1-1/2q}_{pq}(0,T;L^q(\partial\Omega)) \cap L^p (0,T; W^{2-1/q,q}(\partial \Omega))} \Bigg]
		\end{aligned}
	\end{equation}
The constant $\Lambda(\varrho_0,\bu_0, p, q, T)$ depends in general on the ellipticity constant of the differential operator (which depends on $\varrho_0$) and on the amplitude of the lower-order terms (which depend here on $\bu_0$), see, e.g., \cite[Chapter 7]{krylov2024lectures} and \cite{kim2007parabolic}. 

We need to show that we can absorb the terms which depend on $\bu$ by the left-hand side. First, let us start with the term:
\begin{equation}\label{Holder-cont-velocity}
\begin{aligned}
 \int_0^T \|(\bu_0 - \bv ) \cdot \nabla_x \bu\|^p_{L^q(\Omega)}  \dt \leq& \int_0^T \|(\bu_0 - \bv )\|_{L^q(\Omega)}^p \| \nabla_x \bu\|^p_{L^\infty(\Omega)}\dt \\ \leq & C_{\Omega,q}\int_0^T \|(\bu_0 - \bv )\|_{L^q(\Omega)}^p \|\bu\|_{W^{2,q}(\Omega)}^p \dt,
\end{aligned} 
\end{equation}
where $C_{\Omega,q}$ is an embedding constant. 
Now, we use that $\bv \in W^{1,p}(0,T;L^q(\Omega)) \hookrightarrow C^{0,1-1/p}([0,T];L^q(\Omega))$, such that writing 
\[\bu (0) - \bv(t) = \int_0^t \partial_s \bv(s) \,\textup{d}s, \]
and applying H\"older's inequality, we know that: 
\[\|\bu (0) - \bv(t)\|_{L^q(\Omega)} \leq t^{\frac{p-1}{p}} \|\bv\|_{W^{1,p}(0,t;L^q(\Omega))}.\]
Thus, going back to \eqref{Holder-cont-velocity}, we obtain 
\begin{align}
	\int_0^T \|(\bu_0 - \bv ) \cdot \nabla_x \bu\|^p_{L^q(\Omega)}  \dt \leq& C_{\Omega,q}  T^{p-1} \|\bv\|_{W^{1,p}(0,T;L^q(\Omega))}^p \int_0^T  \|\bu\|_{W^{2,q}(\Omega)}^p\dt
\end{align} 
Thus, eventually by reducing the final time $T$ to $T_{\mathrm{eff},1}$, we can absorb the above term by the left side of \eqref{velocity_first_ineq}.
Analogously, we treat the term 
\begin{equation}
	\int_0^T \left\|  \left(\frac1{\varrho} - \frac1{\varrho_0}\right) \div_x (\SS(\bu)) \right\|_{L^q(\Omega)}^p \dt \leq C_{\Omega,q}  T^{p-1} \left\|\frac 1\varrho\right\|_{W^{1,p}(0,T;L^q(\Omega))}^p \int_0^T  \|\bu\|_{W^{2,q}(\Omega)}^p\dt,
\end{equation}
which yields the desired estimate~\eqref{est::prop_lin_velocity}.
\end{proof}

\begin{remark}[Characterization of $F^\kappa_{pq}(0,T;L^q(\partial\Omega))$]
	For a characterization of $F^\kappa_{pq}(0,T;L^q(\partial\Omega))$ used in Proposition~\ref{prop::estimate_lin_velocity_field} for the description of the function space of boundary condition, we refer the reader to \cite[Lemma 6.2]{denk2007optimal} and \cite[Theorem 2.4.1]{triebel1992theory}.
\end{remark}

\subsection{Parabolic estimate for the temperature}
The goal of this section is to establish an $L^p$--$L^q$ estimate for the linearized temperature equation.
Going forward, we fix, without loss of generality, the positive constants
$\kappa,c_v,\xi = 1 $, to simplify notation.

We recall that the linearized (i.e., the velocity field $\bv$ is fixed) initial boundary value problem for the temperature is given by:
\begin{equation}\label{eq::lin_temp}
	\begin{aligned}
		&\pdt \theta + \bv\cdot \nabla_x\theta - \frac{1}{\varrho} \Delta_x \theta = \frac{1}{\varrho} \SS(\bv):\mathbb{D}_x \bv + \frac{1}{\varrho} |\curl_x \bB|^2 - \theta \div_x \bv  \\
		&\theta(0) = \theta_0\\
		& \theta|_{\partial\Omega}=\theta_B
	\end{aligned}
\end{equation}

\begin{proposition}\label{prop::estimate_lin_temp_field}
	Let $T>0$. Let $p,q\in  (1,\infty) \times (d,\infty)$ which verify $p>\frac{2q}{2q-d}$. Let $\varrho \in L^1_{\infty,q}(0,T)$ with $\varrho\geq r_0>0$ on $[0,T]\times \overline\Omega$ and $\varrho(0) = \varrho_0$ on $\Omega$. Let $\bB,\bv \in L^2_{p,q}(0,T)$ with $\bv(0) = \bu_0$ on $\Omega$, $f \in L^p(0,T;L^q(\Omega))$, and let the initial data and boundary data verify
	\eqref{data_temp}.
	
	Then, there exist a unique $\theta \in L^2_{p,q}(0,T) $ which solves \eqref{eq::lin_temp}. Furthermore, there exists a constant $\Lambda$ which may depend on $\varrho_0$, $\bu_0$, $p$, $q$ and $T$ such that 
	\begin{equation}\label{ineq::estimate::temp}
	\begin{aligned}
		\int_0^\tau \big[\| \pdt \theta \|^p_{L^q(\Omega)} + & \|\theta\|_{W^{2,q}(\Omega)}^p \big]\dt   \\ \leq & \Lambda(\varrho_0,\bu_0, p, q,T)  \Bigg[ \left \|\frac{1}{\varrho} \SS(\bv):\mathbb{D}_x \bv \right \|_{L^p(0,\tau; L^q(\Omega))}^p + \left \|\frac{1}{\varrho} |\curl_x \bB|^2 \right \|_{L^p(0,\tau; L^q(\Omega))}^p  \\ & +  \|\theta_0\|_{B^{2(1-1/p)}_{qp}(\Omega)}^p + \| \theta_B\|^p_{F^{1-1/2q}_{pq}(0,\tau;L^q(\partial\Omega)) \cap L^p (0,\tau; W^{2-1/q,q}(\partial \Omega))} 
		\Bigg],
	\end{aligned}
	\end{equation}
holds for all $\tau\in [0,T_{\mathrm{eff},2}\wedge T]$, where $T_{\mathrm{eff},2}>0$ depends on $\Lambda$, $\|\frac1\varrho\|_{W^{1,p}(0,T;L^q(\Omega))}$, $\|\bv\|_{W^{1,p}(0,T;L^q(\Omega))}$, and various embedding constants.
\end{proposition}

\begin{proof}
	
 Let $\tau\in [0,T]$. Similarly to the proof of Proposition~\ref{prop::estimate_lin_velocity_field}, we apply \cite[Theorem 2.3]{denk2007optimal}, which yields the existence of a unique solution $ \theta \in  L^2_{p,q}(0,T)$ to \eqref{eq::lin_temp}, we then use a reformulation with frozen-in-time coefficients ($\frac1\varrho_0$, $\bu_0$) on the left-hand-side of~\eqref{eq::lin_temp}, to deduce the following estimate:
	\begin{equation}\label{temp_first_ineq}
		\begin{aligned}
			\int_0^\tau \big[\| \pdt \theta \|^p_{L^q(\Omega)} + & \|\theta\|_{W^{2,q}(\Omega)}^p \big]\dt   \\ \leq & \Lambda(\varrho_0,\bu_0, p, q,T)  \Bigg[ \int_0^\tau \left\|  \left(\frac1{\varrho} - \frac1{\varrho_0}\right) \Delta_x \theta \right\|_{L^q(\Omega)}^p \dt  +\int_0^\tau \|(\bu_0 - \bv ) \cdot \nabla_x \theta\|^p_{L^q(\Omega)}  \dt\\ &+ \left \|\frac{1}{\varrho} \SS(\bv):\mathbb{D}_x \bv \right \|_{L^p(0,\tau; L^q(\Omega))}^p + \left \|\frac{1}{\varrho} |\curl_x \bB|^2 \right \|_{L^p(0,\tau; L^q(\Omega))}^p  \\ & + \|\theta \div_x \bv \|^p_{L^p(0,\tau; L^q(\Omega))}\\  & +  \|\theta_0\|_{B^{2(1-1/p)}_{qp}(\Omega)}^p + \| \theta_B\|^p_{F^{1-1/2q}_{pq}(0,\tau;L^q(\partial \Omega)) \cap L^p (0,\tau; W^{2-1/q,q}(\partial \Omega))} \Bigg]
		\end{aligned}
	\end{equation}
	Similarly to the proof of Proposition~\ref{prop::estimate_lin_velocity_field}, the terms \[ \int_0^T \left\|  \left(\frac1{\varrho} - \frac1{\varrho_0}\right) \Delta_x \theta \right\|_{L^q(\Omega)}^p \dt  +\int_0^T \|(\bu_0 - \bv ) \cdot \nabla_x \theta\|^p_{L^q(\Omega)} \dt,\]
	are absorbed by exploiting the H\"older continuity of the function $\bv$ and $\frac 1 \varrho$ reducing if needed the final time to an effective final time denoted $T_{\mathrm{eff},2}$.
	
	We also obtain that we can control the term 
	\begin{equation}
		\|\theta \div_x \bv \|_{L^p(0,\tau; L^q(\Omega))} \leq \|\theta\|_{L^\infty(0,\tau;L^q(\Omega))} \|\div_x \bv\|_{L^p(0,\tau;L^\infty(\Omega))},
	\end{equation}
	by, if needed, reducing the final effective time by noticing that 
	\begin{equation}
		\|\theta\|_{L^\infty(0,\tau;L^q(\Omega))} \leq \|\theta_0\|_{L^q(\Omega)}+\|\pdt \theta\|_{L^1(0,\tau; L^q(\Omega))} \leq \|\theta_0\|_{L^q(\Omega)}+\tau^{1-1/p}\|\pdt \theta\|_{L^p(0,\tau; L^q(\Omega))}.
	\end{equation}
	Finally, we show that the assumed regularity on $\frac 1\varrho$, $\bv$, and $\bB$ allows us to control the terms:
	\[\left \|\frac{1}{\varrho} \SS(\bv):\mathbb{D}_x \bv \right \|_{L^p(0,\tau; L^q(\Omega))}^p,\qquad \left \|\frac{1}{\varrho} |\curl_x \bB|^2 \right \|_{L^p(0,\tau; L^q(\Omega))}^p.\]
	To this end, we illustrate our argument on the first term, the second follows similarly. We use the regularity of $\bv$ (see also \eqref{reg_trace_v}) to claim that 
	\begin{equation}\label{eq:embedding_of_interest}
		\SS(\bv) \in C([0,T]; B^{1-2/p}_{qp}(\Omega)) \cap L^p(0,T; W^{1,q}(\Omega)) \hookrightarrow L^{2p}(0,T; B^{1-1/p}_{qp}(\Omega)) \hookrightarrow L^{2p}(0,T; L^{2q}(\Omega)),
	\end{equation}
	where the first embedding is due to the real interpolation with parameters $1/2$ and $p$, the second embedding holds on account of the condition $p>\frac{2q}{2q-d}$; see \cite[Chapter 3]{triebel1983theory}. The same holds true for $\mathbb{D}_x \bv$ which then also belong to the space $L^{2p}(0,T; L^{2q}(\Omega))$. We then conclude that:
	\begin{equation}
		\left \|\frac{1}{\varrho} \SS(\bv):\mathbb{D}_x \bv \right \|_{L^p(0,\tau; L^q(\Omega))}^p \leq \left\|\frac1\varrho \right\|_{L^\infty(0,\tau; L^\infty(\Omega))}^p \|\SS(\bv)\|_{L^{2p}(0,T;L^{2q}(\Omega))}^p \|\mathbb D_x \bv\|_{L^{2p}(0,T;L^{2q}(\Omega))}^p.
	\end{equation}
	Putting the estimates above together we obtain \eqref{ineq::estimate::temp}.
	
\end{proof}
\subsection{Parabolic estimate for the induction equation}
As discussed in \cite[Section 2.2]{feireisl2025conditional}, the magnetic field equation can be transformed into an unconstrained one (i.e., without an explicit constraint $\div_x \bB =0$ on $\Omega$):
\begin{equation}\label{eq::reformulated_magnetics}
	\begin{aligned}
		\pdt \bB - \Delta_x \bB = \curl_x(\bv\times \bB), &\qquad \textrm{on } \ (0,T) \times \Omega, \\
		\div_x \bB = 0, \quad \bv \times n =b_1  &\qquad \textrm{on } \ (0,T) \times \partial\Omega, \\
		\bB (0) = \bB_0 &\qquad \textrm{on } \ \Omega,
	\end{aligned}
\end{equation}
where we set, without loss of generality, the constant coefficient $\xi=1$.
The constraint $\div_x \bB =0$ is directly inherited from the initial condition if $\div_x \bB_0 =0$. 

Kozono and Yaganisawa~\cite[Lemma 4.4]{kozono2009lr} showed that the operator $-\Delta_x$ endowed with the boundary conditions prescribed in \eqref{eq::reformulated_magnetics} is uniformly elliptic and satisfies the Lopatinskii--Shapiro (complementing) condition such that the Agmon--Douglis--Nirenberg theory~\cite{agmon1964estimates} can be applied. It can readily be shown that the combination of $-\Delta_x$ and the boundary conditions in \eqref{eq::reformulated_magnetics} also verifies the Lopatinskii--Shapiro condition for parabolic equations as given in \cite[Section 2]{denk2007optimal}.

\begin{proposition}\label{prop::linearized_mag}
	Let $T>0$. Let $p,q\in  (1,\infty) \times (d,\infty)$. Let $\bv \in L^2_{p,q}(0,T)$ with $\bv(0) = \bu_0$ on $\Omega$ and let the initial data and boundary data verify
	\eqref{data_mag}.
	
	Then the initial-boundary value problem \eqref{eq::reformulated_magnetics} admits a unique solution $\bB \in L^2_{p,q}(0,T)$. Furthermore, there exists a constant $\Lambda$ which may depend on $\bu_0$, $p$, $q$ and $T$  such that
		\begin{equation}\label{ineq::estimate::mag}
		\begin{aligned}
			\int_0^\tau \big[\| \pdt \bB \|^p_{L^q(\Omega)} + & \|\bB\|_{W^{2,q}(\Omega)}^p \big]\dt   \\ \leq & \Lambda(\bu_0, p, q,T)  \Big[  \|\bB_0\|_{B^{2(1-1/p)}_{qp}(\Omega)}^p \\  & + \| b_1\|^p_{F^{1-1/2q}_{pq}(0,T;L^q(\partial\Omega)) \cap L^p (0,T; W^{2-1/q,q}(\partial\Omega)))}, 
			\Bigg]
		\end{aligned}
	\end{equation}
holds for all $\tau \in [0,T\wedge T_{\mathrm{eff},3}]$, where $T_{\mathrm{eff},3}>0$ depends on $\Lambda$, $\|\bv\|_{W^{1,p}(0,T;L^q(\Omega))}$, and various embedding constants.
\end{proposition}
\begin{proof}
	The proof follows along similar ideas to those used to prove Propositions~\ref{prop::estimate_lin_velocity_field} and~\ref{prop::estimate_lin_temp_field}. To avoid unnecessary repetitions we omit the details here.
\end{proof}
\section{Proof of Theorem~\ref{theorem:main_thm}}\label{Sec:Main_proof}
As announced earlier, our goal in this section is to show the local well-posedness of \eqref{eq:energy_equation_theta} supplemented with suitable boundary \eqref{BC:velocity}--\eqref{BC:magneticfield} and initial data \eqref{Initial_data}. To do so, we will use Banach's fixed-point theorem.
Let 
\begin{equation}\label{eq:T_eff_def}
	T_\mathrm{eff} = \min_{1\leq i\leq 3} \{T_{\mathrm{eff},i}\},
\end{equation}
where $T_{\mathrm{eff},i}$ ($1\leq i \leq 3$) are the effective final times from Propositions~\ref{prop::estimate_lin_velocity_field}, \ref{prop::estimate_lin_temp_field}, and \ref{prop::linearized_mag}.
Let $p,q$ satisfy the hypotheses of Lemma~\ref{lemma::est_der}. Consider the closed ball 
\begin{equation}
	\begin{aligned}
	\mathcal{B}(0,\tau) = \Big\{ & \varrho \in L^\infty(0,\tau; W^{1,q}(\Omega)) \cap W^{1,\infty}(0,\tau;L^q(\Omega)), \\&  \bu, \, \theta,\, \bB\, \in L^p(0,\tau; W^{2,q}(\Omega)) \cap W^{1,p}(0,\tau;L^q(\Omega)),\\& 
	\textrm{such that} \quad 
	\varrho(0) = \varrho_0,\, \bu(0) = \bu_0,\, \theta(0)= \theta_0,\,\bB(0)= \bB_0,\\
	& \|\varrho\|_{L^1_{\infty,q}(0,\tau)}\leq K_\varrho, \qquad \|\bu\|_{L^2_{p,q}(0,\tau)}\leq K_{\bu}, \qquad \|\bB\|_{L^2_{p,q}(0,\tau)} \leq K_{\bB}, \qquad \|\theta\|_{L^2_{p,q}(0,\tau)} \leq K_{\theta},\\&
	\varrho \geq r_0>0 \qquad \textrm{on }   [0,\tau]\times\overline{\Omega}  \Big\}. 
	\end{aligned}
\end{equation}
where the initial data $\varrho_0$, $\bu_0$, $\theta_0$, $\bB_0$ are in the spaces prescribed in \eqref{data_density}, \eqref{data_velocity}, \eqref{data_temp}, and \eqref{data_mag}, respectively.
The constants $K_\varrho$, $K_{\bu}$, $K_{\theta}$, $K_{\bB}$ and $r_0$ are positive real numbers to be fixed in the analysis below.

Note that the ball $\mathcal{B}(0,\tau)$ is not empty as the solution to the linear problem belongs to it for $K_\varrho$, $K_{\bu}$, $K_{\bB}$, $K_{\theta}$ large enough. In particular we need $K_{\bu}$ to be large compared to $\|G\|_{L^p(0,T;W^{1,q})}$, see \eqref{est::prop_lin_velocity}.
 
Next, consider the mapping $\mathcal F$, which for  $(\varrho_*, \, \bu_*, \, \theta_*, \, \bB_*) \in \mathcal{B}$ returns the solutions of linearized problems studied above in the following way:
\[\mathcal{F}(\varrho_*, \, \bu_*, \, \theta_*, \, \bB_*) = (\varrho(\bu_*),\, \bu(\varrho_*,\bu_*,\theta_*,\bB_*), \,\theta(\varrho_*,\bu_*,\bB_*), \, \bB(\bu_*) ).\]

\begin{proposition}[Self-mapping of $\mathcal{F}$]\label{prop::self_mapping}
	Let Assumptions~\ref{assu:domain} and \ref{assu:data} hold. Let $p,q$ verify the assumptions of Lemma~\ref{lemma::est_der}. Then, there exist large $K_{\bu}$, $K_{\varrho}$, $K_{\theta}$, $K_{\bB}$ and small $r_0$ and final time $T_*>0$ such that $\mathcal{F}$ is a self-mapping on $\mathcal{B}(0,T_*)$.
\end{proposition}
\begin{proof}
	Let $\tau \in [0,T]$, where $T>0$ is a (large) final time of existence which we will reduce in the course of the proof.
We proceed as follows for the proof of the self-mapping:
\begin{itemize}
\item We start by fixing $K_{\bu}>0$ and $K_\varrho>0$ to be larger than the initial data $\bu_0$ and $\varrho_0$ in the sense of the estimates of Lemmas~\ref{lemma::est_continuity}, \ref{lemma::est_der} and Proposition~\ref{prop::estimate_lin_velocity_field}. E.g., choose 
\begin{equation}\label{def_Ku}
	K_{\bu}^p \geq 2 \Lambda \Bigg[1 + \|G\|^p_{L^p(0,\tau;W^{1,q}(\Omega))}+  \|\bu_0\|_{B^{2(1-1/p)}_{qp}(\Omega)}^p + \| \bu_B\|^p_{F^{1-1/2q}_{pq}(0,\tau;L^q(\partial\Omega)) \cap L^p (0,\tau; W^{2-1/q,q}(\partial \Omega))}\Bigg],
\end{equation}
with $\Lambda$ as in Proposition~\ref{prop::estimate_lin_velocity_field}.

\item We then use the minimum principle \eqref{ineq:max_min} to verify that $\varrho \geq r_0$, eventually by decreasing $r_0$.

Having fixed $K_{\bu}>0$ and $K_\varrho>0$ and $r_0$, the smallest effective final time $T_\mathrm{eff}$ defined in \eqref{eq:T_eff_def} is strictly positive and will be our starting point to find the final time $T_* = \min\{T,T_\mathrm{eff}\}$, which we might need to decrease further.

\noindent Using the estimate of Lemma~\ref{lemma::est_der}, we have control over $|\varrho |_{L^\infty(0,T_*; W^{1,q}(\Omega))}$ in terms of $K_{\bu}$, $\varrho_0$, $\varrho_B$, and final time $T_*$. Recalling that $K_\varrho$ is fixed to be large compared to the initial and boundary data in the sense of estimate~\eqref{estimate_derivative_density}, we can decrease, if necessary, the final time $T_*>0$, so as to ensure that \[\|\varrho\|_{L^1_{\infty,q}(0,T_\mathrm{eff})} \leq K_\varrho.\]

\item Similarly, using the estimate of Proposition~\ref{prop::linearized_mag}, we find a suitable (large) $K_{\bB}$ such that $\|\bB\|_{L^2_{p,q}(0,T_\mathrm{eff})} \leq K_{\bB}$ ;
\item Next, we want to show that there exist $K_{\theta}$ such that we can ensure that $\|\theta\|_{L^2_{p,q}} \leq K_{\theta}$. 
To do so, we notice from estimate~\eqref{ineq::estimate::temp} that the norm of $\theta$ is controlled by the norm of the boundary and initial data. Furthermore, going back to \eqref{eq:embedding_of_interest}, we can modify the embedding such that there exist $1-1/p-d/2q>\varepsilon>0$, and $\delta(\varepsilon,p,q)>0$, such that the following embedding holds 
\begin{equation}\label{embedding_delta_vareps}
	C([0,T]; B^{1-2/p}_{qp}(\Omega)) \cap L^p(0,T; W^{1,q}(\Omega)) \hookrightarrow L^{2p+\delta}(0,T; B^{1-1/p-\varepsilon}_{qp}(\Omega)) \hookrightarrow L^{2p+\delta}(0,T; L^{2q}(\Omega)).
\end{equation}
Thus we can control the term
\begin{equation}
	\begin{aligned}
	&\left \|\frac{1}{\varrho} \SS(\bv):\mathbb{D}_x \bv \right \|_{L^p(0,\tau; L^q(\Omega))}^p \\ 
	& \qquad\leq \tau^{\dfrac{\delta}{(2p+\delta)}} \left\|\frac1\varrho \right\|_{L^\infty(0,\tau; L^\infty(\Omega))}^p \|\SS(\bv)\|_{L^{2p+\delta}(0,\tau;L^{2q}(\Omega))}^p \|\mathbb D_x \bv\|_{L^{2p+\delta}(0,\tau;L^{2q}(\Omega))}^p
	\end{aligned}
\end{equation}
in \eqref{ineq::estimate::temp}, making it arbitrarily smaller for sufficiently small final time (taking into account that the inequality constants can be made non-decreasing in time thanks to Lemma~\ref{lemma:increasing_in_time_embedding}). The same holds true for the term \[ \frac 1\varrho |\curl_x \bB|^2.\]
Thus there exists a combination $K_{\theta}$ and final time (still denoted) $T_*>0$ such that $\|\theta\|_{L^2_{p,q}(0,T_*)} \leq K_{\theta}$;
\item Similarly, for $\bu$, we exploit the regularity of $\theta$, $\bB$, and $\varrho$. In particular, we have 
	\begin{equation}
		\left\|\frac{1}{\varrho} \curl_x\bB \times \bB \right\|_{L^p(0,\tau; L^q(\Omega))}\lesssim \tau^{1/2p} \frac{1}{r_0} \|\curl_x \bB\|_{L^{2p}(0,\tau;L^q(\Omega))} \|\bB\|_{L^\infty(0,\tau;L^\infty(\Omega))},
	\end{equation}
where we made use of the embeddings~\eqref{reg_trace_v}, \eqref{embedding_Besov_cont}, and~\eqref{eq:embedding_of_interest}.
Additionally, we obtain that
\begin{equation}
	\left\|\nabla \theta \right\|_{L^p(0,\tau; L^q(\Omega))}\lesssim \tau^{1/2p}  \|\theta\|_{L^{2p}(0,\tau;L^q(\Omega))},
\end{equation}
and
\begin{equation}
	\left\| \theta \nabla_x \log(\varrho) \right\|_{L^p(0,\tau; L^q(\Omega))}\lesssim \tau^{1/p} \frac{1}{r_0} \|\theta\|_{L^{\infty}(0,\tau;L^\infty(\Omega))} \|\nabla_x\varrho\|_{L^\infty(0,\tau;L^q(\Omega))},
\end{equation}
thus, taking into account~\eqref{def_Ku}, there exists a final time (still denoted) $T_*>0$ such that $\|\bu\|_{L^2_{p,q}(0,T_*)} \leq K_{\bu}$.
\end{itemize}

\end{proof}

Next, we study the contraction properties of the mapping $\mathcal{F}$.

Let $(\varrho^*_1, \, \bu^*_1, \, \theta^*_1, \, \bB^*_1), \ (\varrho^*_2, \, \bu^*_2, \, \theta^*_2, \, \bB^*_2), \ (\varrho_1, \, \bu_1, \, \theta_1, \, \bB_1),\ (\varrho_2, \, \bu_2, \, \theta_2, \, \bB_2) \in \mathcal{B}$  such that:
\begin{equation}
	(\varrho^*_1, \, \bu^*_1, \, \theta^*_1, \, \bB^*_1) = \mathcal{F}(\varrho_1, \, \bu_1, \, \theta_1, \, \bB_1) , \qquad (\varrho^*_2, \, \bu^*_2, \, \theta^*_2, \, \bB_2^*) = \mathcal{F}(\varrho_2, \, \bu_2, \, \theta_2, \, \bB_2).
\end{equation}
We denote their difference by:
\begin{equation}
	\begin{aligned}
		(\overline \varrho^*, \, \overline \bu^*, \,\overline \theta^*, \, \overline \bB^*) &=  (\varrho_1^* - \varrho_2^*, \, \bu_1^* - \bu_2^*, \, \theta_1^* - \theta_2^*, \, \bB_1^* - \bB_2^*),\\
		(\overline \varrho, \, \overline \bu, \,\overline \theta, \, \overline \bB) &=  (\varrho_1 - \varrho_2, \, \bu_1 - \bu_2, \, \theta_1 - \theta_2, \, \bB_1 - \bB_2)
	\end{aligned}
\end{equation}
Then their difference solves the following system of equations:
\begin{equation}\label{sys::difference equation}
	\begin{aligned}
		&\pdt \overline \varrho^* + \div_x (\overline \varrho^* \bu_1 ) = - \div_x(\varrho_2^* \overline \bu)
		\\
		&
		\begin{multlined}
			\pdt \overline\bu^* + \bu_1 \cdot \nabla_x\overline\bu^* -  \frac1{\varrho_1} \div_x (\SS(\overline \bu^*)) =  \left( \frac1{\varrho_1} -\frac1{\varrho_2} \right)\div_x (\SS( \bu_2^*)) + \left(\frac1{\varrho_1} - \frac1{\varrho_2}\right) \curl_x\bB_1 \times \bB_1 \\+ \frac1{\varrho_2}\curl_x \bB_1 \times(\overline \bB) + \frac1{\varrho_2} \curl_x \overline \bB \times \bB_2-\nabla_x \overline\theta - \overline\theta \nabla_x \log(\varrho_1) - \theta_2 \nabla_x \log\left(\frac {\varrho_1 }{\varrho_2}\right) 
		\end{multlined}
		\\
		& \pdt \overline \bB^*  - \Delta_x \overline \bB^*= \curl_x (\bu_1 \times \overline \bB^*) + \curl_x (\overline \bu \times \bB_2^*)\\
		& 
		\begin{multlined}
			\pdt \overline \theta^* + \div_x(\overline\theta^* \bu_1) - \frac1{\varrho_1} \Delta_x \overline \theta^* = - \div_x(\theta^*_2 \overline\bu) + \left( \frac1{\varrho_1} -\frac1{\varrho_2} \right) \Delta_x\theta_2^*+ \left( \frac1{\varrho_1} -\frac1{\varrho_2} \right)  \SS(\bu_1):\mathbb{D}_x \bu_1 \\ + \frac{1}{\varrho_2} \SS(\overline \bu ): \mathbb{D}_x \bu_1 + \frac{1}{\varrho_2} \SS(\bu_2 ): \mathbb{D}_x \overline \bu + \left( \frac1{\varrho_1} -\frac1{\varrho_2} \right)  |\curl_x \bB_1|^2 \\  + \frac{1}{\varrho_2} \left(|\curl_x \bB_1|- |\curl_x \bB_2|\right) \left(|\curl_x \bB_1| + |\curl_x \bB_2|\right) 
		\end{multlined}
	\end{aligned}
\end{equation}
with zero initial and boundary conditions.
\begin{remark}
	The upcoming proof of contractivity will make use of the estimates of Propositions~\ref{prop::estimate_lin_velocity_field}, \ref{prop::estimate_lin_temp_field}, and \ref{prop::linearized_mag}. The difference terms in $\frac 1\varrho$ and $\log\varrho$ will be controlled using the lower bound on $\varrho \geq r_0>0$ on $[0,T]\times \overline \Omega$ so that
	\begin{equation}
		\frac1{\varrho_1} - \frac1{\varrho_2} \leq \frac {\overline{\varrho}}{r_0^2}
	\end{equation}
	and 
	\begin{equation}
		\log\frac{\varrho_1}{\varrho_2} \leq \frac {\overline{\varrho}}{r_0}.
	\end{equation}
On the other hand, the last term in the temperature equation will be handled using the inequality 
\begin{equation} 
	|\curl_x \bB_1|- |\curl_x \bB_2| \leq |\curl_x \overline \bB|.
\end{equation}
\end{remark}

The contractivity of the operator $\mathcal{F}: \mathcal{B} \mapsto \mathcal{B}$, will be given through several lemmas and will be shown in a lower topology compared to the one where we showed existence and well-posedness of solutions. 
The topology used to show contractivity is 
\begin{equation}\label{def:norm_low_top}
	\|\varrho, \bu ,\theta,\bB\|_{\mathcal{B}\mathrm{,low}} = \|\varrho\|_{L^\infty(0,T;L^2(\Omega))}  + \|\bu\|_{L^2(0,T;H^1(\Omega))}  +\|\theta\|_{L^2(0,T;H^1(\Omega))} +\|\bB\|_{L^2(0,T;H^1(\Omega))}.
\end{equation}

Note that it is not unusual in the case of hyperbolic and hyperbolic-parabolic equations which is due to the known problem of loss of regularity for such problems, which makes it impossible to obtain contractivity in the space of solutions. The idea of Kato and Lax as reported in \cite[p. 38]{majda2012compressible} to avoid this difficulty is that it is sufficient to show contractivity in a bigger space (lower topology). We refer to the works of, e.g., \cite{kaltenbacher2022parabolic,kotschote2012strong_allen-cahn,kotschote2012strong} for similar instances of using a lower-topology setting to prove contractivity of the mapping.

First, we consider, similarly to \cite{kotschote2012strong}, the contraction inequality for the continuity equation.

\begin{lemma}\label{lemma:contraction_dens}
%Let $\bu_1,\, \bu_2$ be two velocity fields and let $\varrho_1^*, \,\varrho_2^*$ be two solutions to \eqref{eq:rho} (with $f(\varrho) = 0$) with $\bv = \bu_1$ and $\bv=\bu_2$. Let furthermore the assumptions of Lemma~\ref{lemma::est_der} on $p,\, q$, the regularities of the velocity fields, initial, and boundary data be satisfied. 
Let the assumptions of Proposition~\ref{prop::self_mapping} be satisfied. Let $\tau \in [0,T_*]$, where $T_*$ is as in Proposition~\ref{prop::self_mapping}. Then
	\begin{equation}
		\|\varrho_1^*-\varrho_2^*\|_{L^\infty(0,\tau;L^2(\Omega))} \leq C_* \tau^{1-1/p} \|\bu_1-\bu_2\|_{L^p(0,\tau;H^1(\Omega))},
	\end{equation}
where $C_*$ depends on the ball $\mathcal{B}$ and is increasing with time $T$.
\end{lemma}

\begin{proof}
	We set $\overline \varrho^* = \varrho_1^*-\varrho_2^*$ and set $\overline \bu = \bu_1 -\bu_2$, then $\overline \varrho$ solves
	\begin{equation}\label{eq:rho_contractivity}
		\begin{aligned}
			\pdt \overline \varrho^* + \div_x (\overline \varrho^* \bu_1 ) &= - \div_x(\varrho_2^* \overline \bu) &&\qquad \textrm{on } [0,T] \times \Omega,\\
			\varrho(0,\cdot) &= 0  &&\qquad \textrm{on } \Omega,\\
			\varrho(\cdot,x) &= 0 &&\qquad \textrm{on } \Gamma_{\textup{in}}.
		\end{aligned} 	   
	\end{equation}
We note that we can estimate:
\begin{equation}
	\|\div_x(\varrho_2^* \overline \bu)\|_{L^1(0,T;L^2(\Omega))} \lesssim \|\varrho_2^*\|_{L^\infty(0,T;H^1(\Omega))} \|\overline \bu\|_{L^1(0,T;H^1(\Omega))},
\end{equation}
where we take advantage of the regularity of $\varrho_2$ as ensured by Lemma~\ref{lemma::est_der}.
We can then apply Lemma~\ref{lemma::est_continuity} to \eqref{eq:rho_contractivity} to obtain the estimate:
	\begin{equation}\label{ineq:Lp_rho_contractivity_proof}
	\begin{aligned}
		\|&\overline \varrho^*\|_{L^\infty(0,T;L^2(\Omega))} \leq C_* \tau^{1-1/p}  \|\overline \bu\|_{L^p(0,T;H^1(\Omega))}.
	\end{aligned}
\end{equation}
\end{proof}

We next discuss the contraction inequality of the velocity field. 
\begin{lemma} \label{lemma:contraction_vel}
	Let the assumptions of Proposition~\ref{prop::self_mapping} be satisfied. Let $\tau \in [0,T_*]$, where $T_*$ is as in Proposition~\ref{prop::self_mapping}. Then
 \begin{equation}
 	\begin{aligned}
 	\|\overline\bu^*\|_{L^\infty(0,T;L^2(\Omega))}^2 +  \int_0^\tau \|\overline \bu^*\|^2_{H^1(\Omega)} \dt \lesssim & \tau^{1-1/p} \|\overline \varrho\|_{L^{\infty}(0,\tau; L^2(\Omega))} 
  \\ &	+ \left( \tau^{1/2(1-1/p)}+ \tau^{1/2} \right)\|\overline \bB^* \|_{L^2(0,\tau;H^1(\Omega))} + \tau^{1/2} \|\overline \theta^* \|_{L^2(0,\tau;H^1(\Omega))},
 	\end{aligned}
 \end{equation} 
where the hidden constant depends on the ball $\mathcal{B}$ and is increasing with time $T$.
\end{lemma}
\begin{proof}
	We test the difference equation for the velocity field by $\overline \bu^*$ and integrate by parts in $\Omega$ exploiting the fact that $\overline \bu^*|_{\partial\Omega}=0$:
	\begin{equation} \label{eq:first_eq_temp_contraction}
		\begin{multlined}
		\frac12\pdt \|\overline\bu^*\|_{L^2(\Omega)}^2 + \int_\Omega \bu_1 \cdot \nabla_x\overline\bu^* \overline \bu^* \dx +   \int_\Omega \frac1{\varrho_1} \SS(\overline \bu^*) : \mathbb{D}_x \overline \bu^* \dx = \int_{\Omega} \tilde f \overline \bu^*\dx,
		\end{multlined}
	\end{equation}
	where 
	\begin{equation}
		\begin{aligned}
		\tilde f = &\left( \frac1{\varrho_1} -\frac1{\varrho_2} \right)\div_x (\SS( \bu_2^*)) + \left(\frac1{\varrho_1} - \frac1{\varrho_2}\right) \curl_x\bB_1 \times \bB_1 + \frac1{\varrho_2}\curl_x \bB_1 \times(\overline \bB) + \frac1{\varrho_2} \curl_x \overline \bB \times \bB_2\\ &-\nabla_x \overline\theta - \overline\theta \nabla_x \log(\varrho_1) - \theta_2 \nabla_x \log\left(\frac {\varrho_1 }{\varrho_2}\right) .
		\end{aligned}
	\end{equation}

The second term on the left-hand side of \eqref{eq:first_eq_temp_contraction} can be handled by noticing that for arbitrary $\varepsilon>0$
\begin{equation}
	\int_\Omega \bu_1 \cdot \nabla_x\overline\bu^* \overline \bu^* \dx \leq \varepsilon \|\nabla_x\overline\bu^*\|_{L^2(\Omega)}^2 + C(\varepsilon) \|\bu_1\|_{L^\infty(\Omega)} \|\overline{\bu}^*\|_{L^2(\Omega)}.
\end{equation}
We intend to absorb the $\varepsilon$-term above by the term:
 \begin{equation}
 	\int_\Omega \frac1{\varrho_1} \SS(\overline \bu^*) : \mathbb{D}_x \overline \bu^* \dx \gtrsim \frac 1{K_{\varrho}} \|\overline \bu^*\|^2_{H^1(\Omega)}, 
 \end{equation}
owing to Korn--Poincar\'e's inequality (see, e.g., \cite[Section 11.10]{feireisl2009singular}) and the fact that $(\varrho_1, \cdot,\cdot,\cdot ) \in \mathcal{B}$.
We then integrate over $(0,\tau)$, using Gronwall's inequality to deduce
\begin{equation}
	\begin{aligned}
	\|\overline\bu^*\|_{L^2(\Omega)}^2 +  \int_0^\tau \|\overline \bu^*\|^2_{H^1(\Omega)} \lesssim  &\exp( c   \|\bu_1\|_{L^1(0,\tau;L^\infty(\Omega)})) \|\overline \bu^*\|_{L^\infty(0,\tau;L^\infty(\Omega))} \|\tilde f\|_{L^1(0,\tau;L^1(\Omega))}.
	\end{aligned}
\end{equation}

Next, we note that
\begin{equation}
\begin{aligned}
	\|\tilde f\|_{L^1(0,\tau;L^1)} \lesssim & \tau^{1-1/p} \left(\frac{K_{\bu} + K_{\bB}^2 + K_\theta r_0}{r_0^2} \right)\|\overline \varrho\|_{L^{\infty}(0,\tau; L^2(\Omega))} \\
	& + \frac{K_{\bB}}{r_0} \left( \tau^{1/2(1-1/p)}+ \tau^{1/2} \right)\|\overline \bB^* \|_{L^2(0,\tau;H^1(\Omega))} + \left(1+\frac{K_\varrho}{r_0}\right)\tau^{1/2} \|\overline \theta^* \|_{L^2(0,\tau;H^1(\Omega))},
\end{aligned}
\end{equation}
where $q'$ and $p'$ are the conjugate H\"older exponents of $q$ and $p$ respectively. In particular $q>d>2>q'$.
We have also used the fact that  
\[\curl_x \bB \in L^{2p }(0,\tau;L^q(\Omega)),\]
 which can be be shown using real interpolation of the spaces $L^\infty(0,\tau; B^{1-2/p}_{qp}(\Omega))$ and $L^p(0,\tau; W^{1,q}(\Omega))$ (to which $\curl_x\bB$ belongs) with parameters $1/2$ and $p$.

Putting the estimates together yields the desired estimate.
\end{proof}

The contraction inequalities for the temperature and magnetic fields are obtained analogously. 
\begin{lemma}\label{lemma:contraction_temp}
Let the assumptions of Proposition~\ref{prop::self_mapping} be satisfied. Let $\tau \in [0,T_*]$, where $T_*$ is as in Proposition~\ref{prop::self_mapping}. There exists $\alpha>0$ such that
	\begin{equation}
		\begin{aligned}
			\|\overline\theta^*\|_{L^\infty(0,\tau ;L^2(\Omega))}^2 +   \|\overline \theta^*\|^2_{L^2(0,\tau;H^1(\Omega))} \lesssim & \left( \tau^{1/2(1-1/p)}+ \tau^{1/2} \right)\|\overline \bu^* \|_{L^2(0,\tau;H^1(\Omega))}  
			\\ &	+ ( \tau^{1-1/p} + \tau^\alpha) \|\overline \varrho\|_{L^{\infty}(0,\tau; L^2(\Omega))} \\ &+ \left( \tau^{1/2(1-1/p)} \right)\|\overline \bB^* \|_{L^2(0,\tau;H^1(\Omega))},
		\end{aligned}
	\end{equation} 
where the hidden constant depends on the ball $\mathcal{B}$ and is increasing with time $T$.
\end{lemma}
\begin{proof}
	The proof follows along similar lines to those of Lemma~\ref{lemma:contraction_vel}. We note that we use the embedding ~\eqref{embedding_delta_vareps} and that $\alpha>0$ depends on $\delta$ in \eqref{embedding_delta_vareps}. We omit the rest of the details.
\end{proof}

\begin{lemma}\label{lemma:contraction_mag}
	Let the assumptions of Proposition~\ref{prop::self_mapping} be satisfied. Let $\tau \in [0,T_*]$, where $T_*$ is as in Proposition~\ref{prop::self_mapping}. Then
	\begin{equation}
		\begin{aligned}
			\|\overline\bB^*\|_{L^\infty(0,T;L^2(\Omega))}^2 +  \|\overline \bB^*\|^2_{L^2(0,\tau;H^1(\Omega))} \lesssim  \left( \tau^{1/2(1-1/p)}+ \tau^{1/2} \right)\|\overline \bu^* \|_{L^2(0,\tau;H^1(\Omega))},
		\end{aligned}
	\end{equation} 
where the hidden constant depends on the ball $\mathcal{B}$ and is increasing with time $T$.
\end{lemma}
\begin{proof}
	The proof follows along similar lines to those of Lemma~\ref{lemma:contraction_vel} and is omitted.
\end{proof}

Lemmas~\ref{lemma:contraction_dens},~\ref{lemma:contraction_vel},~\ref{lemma:contraction_temp},~\ref{lemma:contraction_mag} show that we have contractivity of the mapping $\mathcal{F}$ through reducing, if needed, the final time. This finishes the proof of Theorem~\ref{theorem:main_thm}.

\section*{Conclusion}
In this work, we established the local existence of strong solutions in the $L^p$--$L^q$ class to the non-isothermal Navier--Stokes system coupled to a magnetohydrodynamics model in a bounded domain of class $C^2$ in the challenging case of an inflow boundary. We achieved this result by establishing estimates on a linearized system and employing a Banach fixed theorem. Future work will be tasked to identify appropriate blow-up criteria in the spirit of Nash's principle~\cite{nash1958continuity} for the open compressible Navier--Stokes system of equations.
In contrast to the no-inflow case, blow-up criteria are not known for the case $\bu\cdot n \not \geq 0$ on the boundary.

\section*{Declarations}
This work was supported by the Czech Sciences Foundation (GA\v CR), Grant Agreement 24-11034S. The Institute of Mathematics of the Academy of Sciences of the Czech Republic is supported by
RVO:67985840. \\[0mm]

\noindent The author states that there is no conflict of interest.\\[0mm]

\noindent No data are associated with this article

\section*{Acknowledgements}
 The author is grateful to Prof. Eduard Feireisl (AV\v CR) for suggesting the topic of this manuscript and for his helpful discussions, and to Dr. Anna Abbatiello (U Campania) for helpful suggestions. 

\appendix
\section{Appendix}\label{App:Appendix}
We present in this appendix a technical lemma that shows that the embedding constant of 
\[L^2_{pq} \hookrightarrow C([0,T]; B^{2(1-1/p)}_{qp} (\Omega))\] (see, e.g., \cite[Chapter III, Theorem 4.10.2]{amann1995linear}) can be made non-decreasing in (even independent of)  $T$ if we restrict it to the space $\widetilde{L^2_{pq}} = \{\bu \in L^2_{pq} \ | \ \bu(0) = 0\}$.
\begin{lemma}\label{lemma:indep_for_homog}
	The embedding constant of 
	\begin{equation}
		\widetilde{L^2_{pq}(0,T)} \hookrightarrow C([0,T]; B^{2(1-1/p)}_{qp} (\Omega))
	\end{equation}
	is nondecreasing in $T$. More precisely, let $u \in \widetilde{L^2_{pq}}$
	\begin{equation}
		\|u \|_{C([0,T]; B^{2(1-1/p)}_{qp}} \leq c_p \|u \|^{1-1/p}_{L^p(0,T;W^{2,q}(\Omega))}\|u_t \|_{L^{p}(0,T;L^q(\Omega))}^{1/p}.   ,
	\end{equation}
	where $c_p$ is independent of time $T$.
\end{lemma}
\begin{proof}
	Let $u\in L^2_{pq}(0,T)$ and let $\tilde{u} \in L^2_{pq}(\R^+)$ be a regular extension of $u$ to the positive real axis using Lemma~\cite[Lemma 4.10.1]{amann1995linear}. 
	Then \begin{equation}
		\|u(0)\|_{B_{qp}^{2(1-1/p)}(\Omega)} \leq C_{T,p} \|u \|^{1-1/p}_{L^p(0,T;W^2q(\Omega))}\|u \|_{L^p(0,T;W^2q(\Omega))}^{1/p},
	\end{equation}
	where in we have used the convexity of the interpolation inequality \cite[p. 72]{bergh2012interpolation} and the equivalence of the definition of interpolation spaces via ``espaces de moyennes'' and ``espaces de traces''~\cite[Chapter 5]{grisvard1969equations}.
	By translation invariance in $\R^+$, we show that
	\begin{equation}\label{convex}
		\|u\|_{C([0,T];B_{qp}^{2(1-1/p)}(\Omega)} \leq C_{T,p} \|u \|^{1-1/p}_{L^p(0,T;W^{2,q}(\Omega))}\|u \|_{W^{1,p}(0,T;L^q(\Omega))}^{1/p}.
	\end{equation}
	When $u\in \widetilde{L^2_{pq}(0,T)}$, then one can easily see that:
	\begin{equation}\label{eq:elementary}
		\|u \|_{W^{1,p}(0,T;L^q(\Omega))}^{1/p} \sim \|u_t \|_{L^{p}(0,T;L^q(\Omega))}^{1/p}. 
	\end{equation}
	
	To show that the embedding constant can be made independent of time, we use, in what follows, the interval $[0,1]$ as a reference interval and provide a scaling argument. 
	%Let now, $v\in \widetilde{L^2_{pq}(0,1)}$, then 
	%	\begin{equation}
		%	\|v\|_{C([0,1];B_{qp}^{2(1-1/p)}(\Omega)} \leq C_{p} \|v \|^{1-1/p}_{L^p(0,1;W^{2,q}(\Omega))}\|v \|_{W^{1,p}(0,1;L^q(\Omega))}^{1/p},
		%	\end{equation}
	%\end{proof}
	Let now $u\in \widetilde{L^2_{pq}(0,T)}$, and define 
	\begin{equation}
		v: s \in[0,1] \mapsto u(T s),
	\end{equation}
	then  $v\in \widetilde{L^2_{pq}(0,1)}$. Furthermore, elementary calculations yield that
	\begin{equation}
		\begin{aligned}
			v_t(s) &= T u_t(T s),\\
			\|v\|_{L^p(0,1;W^{2,q}(\Omega))} &= T^{-1/p}\|u\|_{L^p(0,T;W^{2,q}(\Omega))},\\
			\|v_t\|_{L^p(0,1,L^q(\Omega))} &= T^{1-1/p} \|u_t\|_{L^p(0,T;L^q(\Omega))},
		\end{aligned}
	\end{equation}
	thus, plugging into \eqref{convex} and using the norm equivalence~\eqref{eq:elementary} yields the desired result.
\end{proof}

We finally, show how Lemma~\ref{lemma:indep_for_homog} can be used to write an inequality with constants that are non-decreasing with time. This was used in the self-mapping and contraction proofs; see, e.g., Proposition~\ref{prop::self_mapping}.

\begin{lemma}\label{lemma:increasing_in_time_embedding} 
	Let $\overline{T}>0$ and let $T\in(0,\overline T].$
	Let $u \in L^2_{pq}(0,T)$
	\begin{equation}
		\|u \|_{C([0,T]; B^{2(1-1/p)}_{qp}} \leq \lambda(\overline T)\|u_0\|_{B^{2(1-1/p)}_{qp}} + c_p \|u \|^{1-1/p}_{L^p(0,T;W^{2,q}(\Omega))}\|u_t \|_{L^{p}(0,T;L^q(\Omega))}^{1/p},
	\end{equation}
	where $c_p$ and  $\lambda$ are independent of $T$.
\end{lemma}
\begin{proof}
	First, we start by extending the initial data to $[0,\overline T ]\times \overline\Omega$ by considering the solution to the parabolic equation:
	\begin{equation}
		\begin{aligned}
			&\pdt w - \Delta_x w = 0 && \textrm{on }\  (0,\overline T) \times \Omega,\\
			&w = 0 && \textrm{on }\  (0,\overline T)\times \partial \Omega,\\
			& w(0)  = u_0&& \textrm{on }\ \Omega .
		\end{aligned}
	\end{equation}
	where $[0,\overline T]$ is an arbitrary finite interval such that $T \leq \overline T$.
	Thus by \cite[Theorem 2.3]{denk2007optimal}, we have the estimate
	\begin{equation}
		\|w\|_{L^2_{pq}} \leq C_{\overline T} \|u_0\|_{B^{2(1-1/p)}_{qp}}.
	\end{equation}
	We apply to the difference $u-w$ Lemma~\ref{lemma:indep_for_homog}, and use the triangle inequality to obtain the desired estimate.
\end{proof}

\end{document}